\documentclass[12pt]{amsart}
\usepackage[style=alphabetic,backend=biber, sorting=nyt]{biblatex}
\bibliography{references}

\usepackage{amssymb}
\usepackage{bbm}
\usepackage{amsmath}
\usepackage{enumitem}
\usepackage{tikz-cd}
\usepackage{faktor}

\usepackage{tikz}
\usetikzlibrary{shapes.geometric}
\tikzset{
v/.style={draw, fill, circle, minimum size=1.5mm, inner sep=0},
u/.style={draw, fill, circle, minimum size=0.3mm, inner sep=0},
b/.style={draw , regular polygon,regular polygon sides=4, minimum size=1.5mm, inner sep=.5mm},
e/.style={very thick},
vs/.style={draw, fill, circle, minimum size=1mm, inner sep=0},
bs/.style={draw,  regular polygon,regular polygon sides=4, minimum size=2mm, inner sep=0mm},
es/.style={thick}
}
\newlength{\nodeheight}
\newlength{\nodewidth}
\usetikzlibrary{arrows,matrix,backgrounds,positioning,patterns}

\pgfdeclarelayer{background}
\pgfsetlayers{background,main}

\usepackage{amsthm}
\theoremstyle{plain}
\newtheorem{theorem}{Theorem}[section]
\theoremstyle{definition} \newtheorem{definition}[theorem]{Definition}
\theoremstyle{plain}
\newtheorem{lemma}[theorem]{Lemma}
\theoremstyle{plain} \newtheorem{proposition}[theorem]{Proposition}
\theoremstyle{plain} 
\theoremstyle{plain} \newtheorem{corollary}[theorem]{Corollary}
\theoremstyle{remark} 
\theoremstyle{remark} \newtheorem{remark}[theorem]{Remark}
\theoremstyle{remark} \newtheorem{example}[theorem]{Example}
\theoremstyle{remark} 
\theoremstyle{definition} 

\providecommand{\varitem}{} 
\makeatletter
\newenvironment{axioms}[1]
 {\renewcommand\varitem[1]{\item[\textbf{#1\arabic{enumi}\rlap{$##1$}.}]%
    \edef\@currentlabel{#1\arabic{enumi}{$##1$}}}%
  \enumerate[label=\textbf{#1\arabic*.}, ref=#1\arabic*]}
 {\endenumerate}
\makeatother

\renewcommand{\t}{\mathbbm{1}}

\newcommand{\Tor}{\mathrm{Tor}}
\newcommand{\Brauer}{\mathrm{Br}}
\newcommand{\TL}{\mathrm{TL}}
\newcommand{\Jones}{\mathrm{J}}
\newcommand{\norm}[1]{\lvert{#1}\rvert}

\newcommand{\Span}{\mathrm{Span}}

\newcommand{\wossit}{naive-cellular}
\newcommand{\sth}{pair-graph}
\newcommand{\linkstate}{link state}
\newcommand{\link}{link}
\newcommand{\Link}{Link}
\newcommand{\axiomletter}{W}

\begin{document}

\title[Homology and cellular algebras]{Homological stability and Graham-Lehrer cellular algebras}

\author{Guy Boyde}
\address{Mathematical Institute, Utrecht University, Heidelberglaan 8
3584 CS Utrecht, The Netherlands}
\email{g.boyde@uu.nl}

\begin{abstract} We show how to formulate some recent results from homological stability of algebras in Graham and Lehrer's language of cellular algebras. The aim is to begin to connect the new results from topology to well-established representation theoretic understandings of these objects.
\end{abstract}

\maketitle

\section{Introduction}

Recently, there has been interest among topologists in certain diagram algebras, such as the Temperley-Lieb \cite{BH,BHComb}, Brauer \cite{BHP}, Iwahori-Hecke \cite{Hepworth,Moselle}, partition \cite{BHP2, Boyde2}, and Jones annular algebras \cite{Boyde2}, as well as equivariant and braided analogues of these \cite{Graves}. These algebras have a long history in representation theory (see for example \cite{GrahamLehrer,KonigXi1,KonigXi2, Xi,HalversonRam,RidoutSaintAubin,BDM}), and Patzt and Sitaraman \cite{Patzt, Sitaraman} have shown that topological language can be used to prove results about representations, so it would be good to show conversely that the representation theoretic language can be used to answer topological questions. The aim of this paper is to make some basic steps in this direction.

Our algebras $A$ (over a commutative unital ground ring $R$) will always be equipped with a trivial module $\t$. For the Brauer and Temperley-Lieb algebras, for example, this is the one-dimensional $R$-module where diagrams with the maximal number of left-to-right connections act by 1, and other diagrams act by zero. By the \emph{homology of $A$} we will always mean the $R$-modules
$$\Tor_q^A(\t,\t)$$
for $q \geq 0$. The primary goal on the topology side is the computation of these modules.

In this paper we attempt to connect this goal to representation theory by giving reasonably general results which describe the homology, with hypotheses in a form familiar to representation theorists.

\subsection{Classes of algebras}

We will define two new classes of algebras, which we call \emph{naive-cellular} and \emph{diagram-like}. Both are inspired by the Graham and Lehrer's \emph{cellular algebras} \cite{GrahamLehrer}, and the definitions are closely related. For a given algebra, we have a diagram of implications:
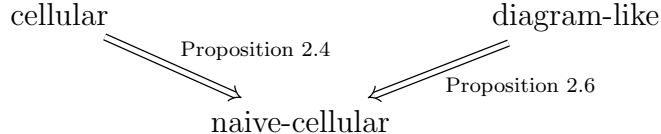
\begin{center}
\begin{tikzcd}
\textrm{cellular} \arrow[dr,Rightarrow,"\textrm{Proposition \ref{prop:cellIsWossit}}"] & & \textrm{diagram-like} \arrow[dl,Rightarrow,"\textrm{Proposition \ref{prop:dgrmLikeIsWossit}}"] \\
& \textrm{naive-cellular} & 
\end{tikzcd}
\end{center}

We introduce these classes essentially because we wish to work in the language of cellular algebras, but view the algebras themselves in a less refined way. In particular, the Robinson-Schensted correspondence seems to be unnecessary for our purposes and too sensitive to the ground ring.

Naive-cellular algebras are essentially defined (Definition \ref{algDef}) by stripping away the Robinson-Schensted correspondence from Graham and Lehrer's definition. What remains is just the idea (familiar from Graham and Lehrer's work) that for example a Brauer diagram may be uniquely decomposed into `two link states and a permutation' or `a left piece, a right piece, and middle piece':

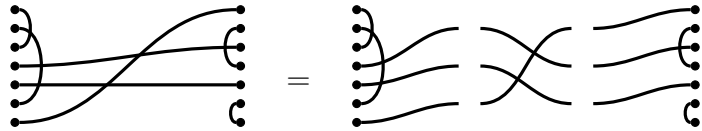
\begin{center}
\begin{tikzpicture}[x=1.5cm,y=-.5cm,baseline=-1.05cm]

\def\wid{2}
\def\hei{0.5}
\def\nodesize{3}
\def\ang{90}

\node[v, minimum size=\nodesize] (a1) at (0* \wid,0*\hei) {};
\node[v, minimum size=\nodesize] (a2) at (0* \wid,1*\hei) {};
\node[v, minimum size=\nodesize] (a3) at (0* \wid,2*\hei) {};
\node[v, minimum size=\nodesize] (a4) at (0* \wid,3*\hei) {};
\node[v, minimum size=\nodesize] (a5) at (0* \wid,4*\hei) {};
\node[v, minimum size=\nodesize] (a6) at (0* \wid,5*\hei) {};
\node[v, minimum size=\nodesize] (a7) at (0* \wid,6*\hei) {};

\node[v, minimum size=\nodesize] (d1) at (1* \wid,0*\hei) {};
\node[v, minimum size=\nodesize] (d2) at (1* \wid,1*\hei) {};
\node[v, minimum size=\nodesize] (d3) at (1* \wid,2*\hei) {};
\node[v, minimum size=\nodesize] (d4) at (1* \wid,3*\hei) {};
\node[v, minimum size=\nodesize] (d5) at (1* \wid,4*\hei) {};
\node[v, minimum size=\nodesize] (d6) at (1* \wid,5*\hei) {};
\node[v, minimum size=\nodesize] (d7) at (1* \wid,6*\hei) {};

\draw[e] (a1) to [out=0,in=0] (a3);
\draw[e] (a2) to [out=0,in=0] (a6);

\draw[e] (a4) to [out=0,in=180] (d3);
\draw[e] (a5) to [out=0,in=180] (d5);
\draw[e] (a7) to [out=0,in=180] (d1);

\draw[e] (d2) to [out=180,in=180] (d4);
\draw[e] (d6) to [out=180,in=180] (d7);

\end{tikzpicture}
\quad
$=$
\quad
\begin{tikzpicture}[x=1.5cm,y=-.5cm,baseline=-1.05cm]

\def\wid{2}
\def\hei{0.5}
\def\nodesize{3}
\def\ang{90}

\node[v, minimum size=\nodesize] (a1) at (0.5* \wid,0*\hei) {};
\node[v, minimum size=\nodesize] (a2) at (0.5* \wid,1*\hei) {};
\node[v, minimum size=\nodesize] (a3) at (0.5* \wid,2*\hei) {};
\node[v, minimum size=\nodesize] (a4) at (0.5* \wid,3*\hei) {};
\node[v, minimum size=\nodesize] (a5) at (0.5* \wid,4*\hei) {};
\node[v, minimum size=\nodesize] (a6) at (0.5* \wid,5*\hei) {};
\node[v, minimum size=\nodesize] (a7) at (0.5* \wid,6*\hei) {};

\node[] (b1) at (1* \wid,1*\hei) {};
\node[] (b2) at (1* \wid,3*\hei) {};
\node[] (b3) at (1* \wid,5*\hei) {};

\node[] (c1) at (1.5* \wid,1*\hei) {};
\node[] (c2) at (1.5* \wid,3*\hei) {};
\node[] (c3) at (1.5* \wid,5*\hei) {};

\node[v, minimum size=\nodesize] (d1) at (2* \wid,0*\hei) {};
\node[v, minimum size=\nodesize] (d2) at (2* \wid,1*\hei) {};
\node[v, minimum size=\nodesize] (d3) at (2* \wid,2*\hei) {};
\node[v, minimum size=\nodesize] (d4) at (2* \wid,3*\hei) {};
\node[v, minimum size=\nodesize] (d5) at (2* \wid,4*\hei) {};
\node[v, minimum size=\nodesize] (d6) at (2* \wid,5*\hei) {};
\node[v, minimum size=\nodesize] (d7) at (2* \wid,6*\hei) {};

\draw[e] (b1) to [out=0,in=180] (c2);
\draw[e] (b2) to [out=0,in=180] (c3);
\draw[e] (b3) to [out=0,in=180] (c1);

\draw[e] (a1) to [out=0,in=0] (a3);
\draw[e] (a2) to [out=0,in=0] (a6);

\draw[e] (a4) to [out=0,in=180] (b1);
\draw[e] (a5) to [out=0,in=180] (b2);
\draw[e] (a7) to [out=0,in=180] (b3);

\draw[e] (d2) to [out=180,in=180] (d4);
\draw[e] (d6) to [out=180,in=180] (d7);

\draw[e] (c1) to [out=0,in=180] (d1);
\draw[e] (c2) to [out=0,in=180] (d3);
\draw[e] (c3) to [out=0,in=180] (d5);

\end{tikzpicture}
\end{center}

For some of what we want to do, this definition is not rigid enough, so we introduce diagram-like algebras (Definition \ref{def:dgrmLike}). Roughly, this is an algebra with a basis which can be decomposed into `a left piece, a right piece, and middle piece', as before, but now the product of two basis elements must be a multiple of another basis element, and the decomposition of the result must depend in a specific (fairly obvious) way on the decomposition of the factors. We will show that the Brauer, Jones annular, and Temperley-Lieb algebras are diagram-like, hence naive-cellular. The facts we need are essentially already written down in Graham and Lehrer's paper. Xi \cite{Xi} has shown that the partition algebras are cellular, and they are almost certainly diagram-like. However, Xi takes a different approach to Graham and Lehrer, and this paper is already quite long, so we chose to omit the partition algebras, rather than spend more pages proving the necessary basic facts.

One concrete sense in which cellular algebras are `too refined' is as follows: the results we wish to prove sometimes hold in situations where the algebra in question is not cellular. For example, we will use our methods to provide an alternate proof Theorem 1.5 of \cite{Boyde2}, which says that the homology of the Jones annular algebra $\Jones_n(\delta)$ coincides with that of the cyclic group $C_n$ whenever $n$ is odd or $\delta$ is invertible, but Graham and Lehrer's proof that $\Jones_n(\delta)$ is cellular assumes the splitting of a certain polynomial in the ground ring \cite[Theorem 6.15]{GrahamLehrer}.

A word of caution: although cellular algebras are naive-cellular, we will not in general be working with the naive-cellular structure obtained from Graham and Lehrer's cellular structure. This happens to be true for the Temperley-Lieb algebras (because `the only planar permutation is the identity') but for the Brauer and Jones annular algebras we use a different structure (albeit one which shows up in Graham and Lehrer's work - see Remark \ref{rmk:NotTheSame}).

\subsection{Overview of results}

Our first main result, Theorem \ref{thm:global}, is for computing all of the homology of an algebra. We frame it for naive-cellular algebras, and is intended to illustrate the connection to Graham and Lehrer's language, thereby making clear that the hypotheses are quite mild. As an example, we will use it to recover the `global' results on the Jones annular algebras $\Jones_n(\delta)$ from \cite{Boyde2}, taking the opportunity to give an entertaining alternative proof via `twisting' that would not have been possible for the Temperley-Lieb algebras. It could equally well have been used to recover the known global results on the Temperley-Lieb \cite{BH,Sroka}, Brauer \cite{BHP}, or Partition algebras \cite{BHP2}, but we will restrict ourselves to a single example. In Section \ref{subsection:globalResults}, we will state this theorem, and then in Section \ref{subsection:Cellular} we will give the special case of that statement for cellular algebras. These statements closely resemble some of Graham and Lehrer's \cite{GrahamLehrer}.

It is also desirable to make some connection to stability, which is to say, results which hold under milder hypotheses, but only in a range of homological degrees. From a topological point of view, these are the more interesting results. We will prove a technical proposition (Proposition \ref{prop:Idempotents}) and show how to use it in tandem with the main theorem of our previous paper \cite{Boyde2} to prove basic stability results. This is where we introduce diagram-like algebras. Roughly speaking, we want to think of Proposition \ref{prop:Idempotents} as the representation-theoretic `front-end' of this pipeline, and the main theorem from \cite{Boyde2} as the topological `back-end'. As an example, we will recover Sroka's homological stability and vanishing result for the Temperley-Lieb algebras. Again, we could equally well have recovered the known stability results for the Temperley-Lieb \cite{BH,Sroka}, Jones annular \cite{Boyde2}, or Partition algebras \cite{BHP2}, but we will restrict ourselves to a single example.

The hope is that by running the machine consisting of Proposition \ref{prop:Idempotents} and the main theorem of \cite{Boyde2}, one puts the topology in a black box, so that users not familiar with e.g. spectral sequences can easily establish basic stability results. We stress that stability results obtained in this way will not in general be optimal: Boyd-Hepworth's original proof of stability for the Temperley-Lieb algebras obtains a better range, of slope $1$, and requires much more intricate work. That said, for the partition algebras, one can obtain the optimal stability range due to Boyd-Hepworth-Patzt \cite{BHP2} in this way \cite{Boyde2}.

\subsection{Global results} \label{subsection:globalResults}

Readers familiar with cellular algebras may wish to first skip to Section \ref{subsection:Cellular}, where we give a version in that context. We will say that a subset $X$ of a poset $\Lambda$ is \emph{downward closed} if whenever $\mu \leq \lambda$ and $\lambda \in X$ then $\mu \in X$. Many of the symbols appearing here will not be defined until later: for naive-cellular algebras see Definition \ref{algDef}, for the modules $W(\lambda)$ see Definition \ref{def:LinkMod}, for the bilinear forms $\langle \phantom{x},\phantom{x} \rangle_{\tau}$ see Definition \ref{def:innerProduct}, and for the ideals $I_X$ see Definition \ref{def:I}.

\begin{theorem} \label{thm:global} Let $A$ be a {\wossit} algebra over a ring $R$, with {\wossit} datum $(\Lambda, G, M, C, *)$. Let $Y \subset X$ be downward closed subsets of $\Lambda$, with $X \setminus Y$ finite. Suppose that for each $\lambda \in X \setminus Y$ we have the following. For every $q \in M(\lambda)$, there exists $v \in W(\lambda)$ such that for some $\sigma \in G(\lambda)$ we have $$\langle C_q, v \rangle_{\tau} = \begin{cases} 1 & \textrm{ if } \tau=\sigma, \textrm{ and} \\ 0 & \textrm{ otherwise.} \end{cases}$$
Then, for any right $A$-module $M$ and left $A$-module $N$, where $I_X$ acts trivially on both, we have
$$\Tor_*^{\faktor{A}{I_{Y}}}(M,N) \cong \Tor_*^{\faktor{A}{I_{X}}}(M,N).$$ \end{theorem}

We will prove this theorem in Section \ref{section:GlobalResults}. In the case of cellular algebras, our definitions coincide with Graham and Lehrer's, and their results suggest that the hypotheses of this theorem are not too strong (see Section \ref{subsection:Cellular}).

\subsection{Results which hold only in a range}

In our paper \cite{Boyde2} we made the following definition and proved the following theorem. Write $\underline{w}$ for the set $\{1,2,\dots, w\}$.

\begin{definition} Let $A$ be an $R$-algebra, let $I$ be a twosided ideal of $A$, and let $w \geq h \geq 1$. An \emph{idempotent (left) cover} of $I$ of \emph{height $h$} and \emph{width $w$} is a finite collection of left ideals $J_1, \dots J_w$ of $A$, which cover $I$ in the sense that $J_1 + \dots + J_w = I$, and such that for each $S \subset \underline{w}$ with $\norm{S} \leq h$, the intersection $$\bigcap_{i \in S} J_i$$ is either zero or is a principal left ideal generated by an idempotent. If $I$ is free as an $R$-module, then an idempotent cover is said to be \emph{$R$-free} if there is a choice of $R$-basis for $I$ such that each $J_i$ is free on a subset of this basis.
\end{definition}

\begin{theorem} \label{oldThm} Let $A$ be an augmented $R$-algebra with trivial module $\t$. Let $I$ be a twosided ideal of $A$ which is free as an $R$-module and acts trivially on $\t$. Suppose that there exists an $R$-free idempotent left cover of $I$ of height $h$. Then the natural map $$\Tor_q^A(\t,\t) \to \Tor_q^{\faktor{A}{I}}(\t,\t)$$ is an isomorphism for $q \leq h - 2$, and a surjection for $q = h-1$. If $h$ is equal to the width $w$ of the cover, then this map is an isomorphism for all $q$.
\end{theorem}

In that paper, we constructed covers for standard ideals of the Jones annular and partition algebras, computed their height, and applied the theorem to obtain stable homology isomorphisms. The most challenging step of this process is to compute the height (i.e. to show that certain ideals are principal and generated by idempotents), and the point of the second half of this paper is that the (technical) Proposition \ref{prop:Idempotents} can be used to put this computation into cellular-type language. As a proof of concept, we will use this method to recover Sroka's results on the homology of the Temperley-Lieb algebras.

\subsection{The cellular case of Theorem \ref{thm:global}} \label{subsection:Cellular}

Consider Theorem \ref{thm:global}. In the cellular case, each $G(\lambda)$ is trivial (Proposition \ref{prop:cellIsWossit}), and the single bilinear form $\langle \phantom{x}, \phantom{y} \rangle_1$ corresponding to the identity element is their bilinear form $\phi_\lambda$ (c.f. Remarks \ref{rmk:cellMod} and \ref{rmk:cellForm}). The resulting special case of Theorem \ref{thm:global} is as follows:

\begin{theorem} \label{thm:globalCellular} Let $A$ be a cellular algebra over a ring $R$, with cell datum $(\Lambda, M, C, *)$. Let $Y \subset X$ be downward closed subsets of $\Lambda$, with $X \setminus Y$ finite. Suppose that for each $\lambda \in X \setminus Y$ we have the following. For every $q \in M(\lambda)$, there exists $v \in W(\lambda)$ such that $$\phi_\lambda(C_q, v) = 1.$$
Then, for any right $A$-module $M$ and left $A$-module $N$, where $I_X$ acts trivially on both, we have
$$\Tor_*^{\faktor{A}{I_{Y}}}(M,N) \cong \Tor_*^{\faktor{A}{I_{X}}}(M,N).$$ \end{theorem}

In words, the condition of the theorem asks that for every member $C_q$ of the canonical basis of $W(\lambda)$, the image of the function $W(\lambda) \to R$ given by $v \mapsto \phi_\lambda(C_q,v)$ is all of $R$. If $R$ is a field, then this is equivalent to asking that the elements $C_q$ do not lie in the radical of $\phi_\lambda$.

Graham and Lehrer show \cite[Theorem 7.3]{GrahamLehrer} that, over a field, semisimplicity of $A$ is equivalent to all of the bilinear forms $\phi_\lambda$ being nondegenerate, which is a much stronger condition. Since their condition is defined $\lambda$-wise, this is actually equivalent to semisimplicity of the quotients $\faktor{A}{I_X}$ for each $X$. In this case, the above theorem is trivial: if $\faktor{A}{I_X}$ and $\faktor{A}{I_Y}$ are both semisimple, then all exact sequences of modules over these two algebras are split, so all of their finitely generated modules are projective, and both of the $\Tor$ groups appearing in Theorem \ref{thm:globalCellular} vanish automatically.

\subsection{The structure of this paper}

In Section \ref{section:Classes} we will define naive-cellular and diagram-like algebras (Definitions \ref{algDef} and \ref{def:dgrmLike}), and prove that diagram like algebras are naive-cellular (Proposition \ref{prop:dgrmLikeIsWossit}), and that cellular algebras are naive-cellular (Proposition \ref{prop:cellIsWossit}). We will then show (Section \ref{section:Ex}) that the Brauer, Jones annular, and Temperley-Lieb algebras are diagram-like. We repeat the caution that the naive-cellular structures on these algebras that we will work with are not the same as the ones obtained from Graham and Lehrer's cellular structures.

In Section \ref{section:BasicProperties}, we develop the basic theory. This follows Graham and Lehrer's development quite closely, with no important surprises. The basic theory culminates with the definition of the link modules $W(\lambda)$ and the associated bilinear form, which again is a direct analogue of Graham and Lehrer's definition.

At this stage, we have the language to prove Theorem \ref{thm:global}, and we do so in Section \ref{section:GlobalResults}. We specialise this theorem to subalgebras of the Brauer algebras in Section \ref{section:Specialisation}, and in Section \ref{section:Jones} we use the specialised theorem to recover the odd strand and invertible parameter results for the Jones annular algebra proven in \cite{Boyde2} from a different perspective.

To prove stability results, we need a bit more. In Section \ref{section:LSO} we define the notion of a \emph{link state ordering} (Definition \ref{def:LSOrdering}) and we use it to prove our second main theoretical result, Proposition \ref{prop:Idempotents}. In the final section (Section \ref{section:HS}), we use Proposition \ref{prop:Idempotents} together with the main theorem of \cite{Boyde2} (Theorem \ref{oldThm}) to give a different perspective on Sroka's proof of homological stability for Temperley-Lieb algebras.

\subsection{Omissions and further directions}

There are many things that we do not discuss here, but which might be good directions for further work.

Patzt \cite{Patzt} and Sitaraman \cite{Sitaraman} have already studied representation stability (\'a la Church-Ellenberg-Farb \cite{ChurchEllenbergFarb}) for the Temperley-Lieb algebras, and Patzt also treats the Brauer and partition algebras. We wonder if there is a useful way to take a cellular point of view on these results.

There is another point of view on cellular algebras, due to K\"onig and Xi \cite{KonigXi1,KonigXi2}. Some of what we do here closely resembles some of what they did, and it would be interesting to see if there is any connection.

\subsection*{Acknowledgements}

I would like to thank Richard Hepworth for his encouragement, and for providing the original impetus for this project, by telling me about Graham and Lehrer's paper \cite{GrahamLehrer}, and for suggesting that it might provide the correct language for the results of my paper \cite{Boyde}. I would also like to acknowledge the substantial technical debt to Graham and Lehrer. This work was supported by the European Research council (ERC) through Gijs Heuts' grant “Chromatic homotopy theory of spaces”, grant no. 950048.

\section{Classes of algebras} \label{section:Classes}

In this section, we define {\wossit} and diagram-like algebras, recalling first the definition of cellular algebras.

\subsection{Cellular algebras}

Throughout, $R$ will be a commutative ring with unit. In \cite{GrahamLehrer}, the authors define a class of algebras that they call \emph{cellular}:

\begin{definition}[{\cite[Definition 1.1]{GrahamLehrer}}] \label{cellAlgDef} A \emph{cellular algebra} over $R$ is an associative unital algebra $A$, together with \emph{cell datum} $(\Lambda, M, C, *)$, where
\begin{axioms}{C}
  \item \label{item:C1} $\Lambda$ is a poset, and for each $\lambda \in \Lambda$, $M(\lambda)$ is a finite set (the set of \emph{tableaux of type $\lambda$}) such that \begin{align*}
      \coprod_{\lambda \in \Lambda} M(\lambda) \times M(\lambda) & \to A \\
      (S,T) & \mapsto C_{S,T}^{\lambda}
  \end{align*}
  is an injective map with image an $R$-basis of $A$.
  \item \label{item:C2} The map $*$ is an $R$-linear anti-involution of $A$ such that $$(C_{S,T}^{\lambda})^*=C_{T,S}^{\lambda}.$$
  \item \label{item:C3} If $\lambda \in \Lambda$ and $S,T \in M(\lambda)$ then for any element $a \in A$ we have $$a C_{S,T}^{\lambda} \equiv \sum_{S' \in M(\lambda)} r_a(S',S) C_{S',T}^{\lambda} (\mathrm{mod} A(< \lambda)),$$ where $r_a(S',S) \in R$ is independent of $T$ and where $A(< \lambda)$ is the $R$-submodule of $A$ generated by the elements $\{C_{A,B}^{\mu} \lvert  A,B \in M(\mu), \mu < \lambda\}$.
\end{axioms}
\end{definition}

\subsection{Naive-cellular algebras}

Our main class of algebra will be defined as follows.

\begin{definition} \label{algDef} A \emph{{\wossit} algebra} over $R$ is an associative unital algebra $A$, together with \emph{{\wossit} datum} $(\Lambda, G, M, C, *)$, where
\begin{axioms}{\axiomletter}
  \item \label{item:W1} $\Lambda$ is a poset. For each $\lambda \in \Lambda$, $G(\lambda)$ is a group, and $M(\lambda)$ is a finite set (the set of \emph{{\linkstate}s of type $\lambda$}) such that \begin{align*}
      \coprod_{\lambda \in \Lambda} M(\lambda) \times G(\lambda) \times M(\lambda) & \to A \\
      (p,\sigma,q) & \mapsto C_{p,q}^{\sigma}
  \end{align*}
  is an injective map with image an $R$-basis of $A$.
  \item \label{item:W2} The map $*$ is an $R$-linear anti-involution of $A$ such that $$(C_{p,q}^{\sigma})^*=C_{q,p}^{\sigma^{-1}}.$$
  \item \label{item:W3} If $\sigma \in G(\lambda)$ and $p,q \in M(\lambda)$ then for any element $a \in A$ we have $$a C_{p,q}^{\sigma} \equiv \sum_{\substack{\sigma' \in G(\lambda) \\ p' \in M(\lambda) }} r_a(p',\sigma'\sigma^{-1},p) C_{p',q}^{\sigma'} (\mathrm{mod} I_{< \lambda}),$$ where $r_a(p',\sigma'\sigma^{-1},p) \in R$ depends only on $a$, $p'$, $p$, and the product $\sigma'\sigma^{-1}$, and where $I_{< \lambda}$ is the $R$-submodule of $A$ generated by the elements $\{C_{p'',q''}^{\tau} \lvert  \mu < \lambda, \tau \in G(\mu), p'',q'' \in M(\mu)\}$.
\end{axioms}
\end{definition}

\begin{remark} \label{rmk:W3} Much as for cellular algebras, applying the involution $*$ to the equation \ref{item:W3}, we get, for any $a \in A$, $\lambda \in \Lambda$, $\sigma \in G(\lambda)$ and $p,q \in M(\lambda)$, \begin{equation*} C_{q,p}^{\sigma^{-1}}a^* \equiv \sum_{\substack{\sigma' \in G(\lambda) \\ p' \in M(\lambda)}} r_a(p',\sigma'\sigma^{-1},p)C_{q,p'}^{(\sigma')^{-1}} (\mathrm{mod} I_{< \lambda}).
\end{equation*}
Here we have used the fact that Axiom \ref{item:W2} implies that $I_{< \lambda}^*=I_{< \lambda}$. For later use, we record the more useful form of this equation obtained by changing variables:
\begin{equation*} C_{p,q}^{\tau}b \equiv \sum_{\substack{\tau' \in G(\lambda) \\ q' \in M(\lambda)}} r_{b^*}(q',(\tau')^{-1}\tau,q)C_{p,q'}^{\tau'} (\mathrm{mod} I_{< \lambda}).
\end{equation*}
\end{remark}

It follows immediately from the definitions that {\wossit} algebras do indeed generalise cellular algebras:

\begin{proposition} \label{prop:cellIsWossit} The map $$(\Lambda, M, C, *) \mapsto (\Lambda, 1, M, C, *),$$ given by assigning the trivial group to each $\lambda \in \Lambda$, carries a cell datum to a {\wossit} datum, and identifies cellular algebras with those {\wossit} algebras where each group $G(\lambda)$ is the trivial group. \qed
\end{proposition}

\subsection{Diagram-like algebras}

\begin{definition} \label{def:dgrmLike} A \emph{diagram-like algebra} over $R$ is defined identically to a {\wossit} algebra, except that instead of Axiom \ref{item:W3} of that definition, we require the following condition on the product of two basis elements:

For any $\lambda_1, \lambda_2 \in \Lambda$, $\sigma_1 \in G(\lambda_1), \sigma_2 \in G(\lambda_2)$, $p_1, q_1 \in M(\lambda_1)$ and $p_2,q_2 \in M(\lambda_2)$, we have:
$$C_{p_1,q_1}^{\sigma_1} C_{p_2,q_2}^{\sigma_2} = \kappa C_{p,q}^{\sigma},$$ where the objects ($\kappa,p,q,\sigma$) on the right are implicitly functions of those on the left, such that (letting $\lambda \in \Lambda$ be the element corresponding to $C_{p,q}^{\sigma}$) we have:
\begin{enumerate}
    \item \label{item:dl1} $\lambda \leq \lambda_1$, and $\lambda \leq \lambda_2$ 
    \item \label{item:dl2} $\lambda$ and $\kappa$ depend only on $q_1$ and $p_2$
    \item \label{item:dl3} $\sigma$ depends only on $\sigma_1, q_1, p_2,$ and $\sigma_2$
    \item \label{item:dl4} $p$ depends only on $p_1,\sigma_1,q_1,$ and $p_2$.
    \item \label{item:dl5} If $\lambda = \lambda_2$, then $q = q_2$, and $\sigma \sigma_2^{-1}$ depends only on $\sigma_1, q_1, $ and $p_2$.
\end{enumerate}
\end{definition}

We will call Conditions (\ref{item:dl1}) - (\ref{item:dl5}) the \emph{dependency conditions}.


\begin{proposition} \label{prop:dgrmLikeIsWossit} A diagram-like algebra is in particular a {\wossit} algebra.
\end{proposition}

\begin{proof} The two definitions differ only in the replacement of the third axiom (\ref{item:W3}) defining a {\wossit} algebra by the condition of Definition \ref{def:dgrmLike}. It therefore suffices to show that, together with the first two axioms defining a {\wossit} algebra, this condition implies the third axiom.

The algebra multiplication is $R$-bilinear, so in particular it is $R$-linear in the first variable, and $a \in A$ is of course uniquely expressible in the basis. It therefore suffices to verify this third axiom in the case that $a = C_{p_1,q_1}^{\sigma_1}$ is a basis element.

Thus, consider a product $a C_{p_2,q_2}^{\sigma_2} = C_{p_1,q_1}^{\sigma_1} C_{p_2,q_2}^{\sigma_2}$.

The idea is to transform the problem into one about indicator functions. Definition \ref{def:dgrmLike} says first of all that $$a C_{p_2,q_2}^{\sigma_2} = C_{p_1,q_1}^{\sigma_1} C_{p_2,q_2}^{\sigma_2} = \kappa C_{p,q}^{\sigma}.$$

By Dependency Condition (\ref{item:dl1}) of the condition, the product lies in $I_{\leq \lambda_2}$, so the above equation may be rewritten as $$C_{p_1,q_1}^{\sigma_1} C_{p_2,q_2}^{\sigma_2} = \sum_{\lambda' \leq \lambda_2} \sum_{\substack{\sigma' \in G(\lambda') \\ p' \in M(\lambda')}} \kappa(q_1,p_2)\chi_{\sigma}(\sigma')\chi_p(p') C_{p',q_2}^{\sigma'}$$ where $\chi_\sigma$ and $\chi_p$ are the indicator functions of $\sigma$ and $p$ respectively.

By Dependency Condition (\ref{item:dl2}), there exists a function $\overline{\kappa}:M(\lambda_1) \times M(\lambda_2) \to R$ defined by setting $\overline{\kappa}(q_1,p_2) = \begin{cases} \kappa & \lambda = \lambda_2 \\
0 & \lambda < \lambda_2.
\end{cases}$

Modulo $I(< \lambda_2)$, we may therefore write:
$$a C_{p_2,q_2}^{\sigma_2} = C_{p_1,q_1}^{\sigma_1} C_{p_2,q_2}^{\sigma_2} \equiv \sum_{\substack{\sigma' \in G(\lambda_2) \\ p' \in M(\lambda_2)}} \overline{\kappa}(q_1,p_2)\chi_{\sigma}(\sigma')\chi_p(p') C_{p',q_2}^{\sigma'},$$ and what must be established is that this is of the correct form for Axiom \ref{item:W3}, namely that $\overline{\kappa}(q_1,p_2)\chi_{\sigma}(\sigma')\chi_p(p') \in R$ depends only on $a$, $\sigma'\sigma_2^{-1},$ $p_2$, and $p'$. It suffices to verify this factor-by factor.

First $\overline{\kappa}(q_1,p_2)$ depends only on $q_1$ and $p_2$ (Dependency Condition (\ref{item:dl2})) which is to say only on $a$ and $p_2$, as required. This implies the result when $\overline{\kappa}(q_1,p_2) = 0$, so it suffices to establish it in the case $\overline{\kappa}(q_1,p_2) \neq 0$, which means we may henceforth assume $\lambda=\lambda_2$.

On the face of it, $\chi_{\sigma}(\sigma')$ depends on $\sigma$ and $\sigma'$, but $\sigma$ is implicitly (Dependency Condition (\ref{item:dl3})) a function of $a$, $p_2$, and $\sigma_2$. We must show that this reduces the dependency of $\chi_{\sigma}(\sigma')$ on $\sigma$ and $\sigma'$ to a dependency on $a$, $p_2$, and $\sigma'\sigma_2^{-1}$. Using the fact that $\chi_{\sigma}$ is an indicator function on a group, we get $$\chi_{\sigma}(\sigma') = \chi_{\sigma}(\sigma'\sigma_2^{-1}\sigma_2) = \chi_{\sigma \sigma_2^{-1}}(\sigma'\sigma_2^{-1}).$$ By Dependency Condition (\ref{item:dl5}), $\sigma \sigma_2^{-1}$ depends only on $a$ and $p_2$, so, $\chi_{\sigma}(\sigma')$ depends only on $a$, $p_2$, and $\sigma'\sigma_2^{-1}$, as required.

Dependency Condition (\ref{item:dl4}) then immediately gives that $\chi_p(p')$ depends only on $a$, $p_2$, and $p'$, as required.
\end{proof}

\subsection{Restriction of link states and groups}

The following ideas will be useful for verifying examples.

\begin{definition} \label{def:restriction} Let $A$ be a diagram-like algebra, with {\wossit} datum $(\Lambda,G,M,C,*)$. A subalgebra $A'$ of $A$ will be said to be \emph{obtained by restriction} if for each $\lambda$ in $\Lambda$ there exists a (perhaps empty) subset $M'(\lambda)$ of $M(\lambda)$ and a subgroup $G'(\lambda)$ of $G(\lambda)$ such that \begin{align*}
      \coprod_{\lambda \in \Lambda} M'(\lambda) \times G'(\lambda) \times M'(\lambda) & \to A \\
      (p,\sigma,q) & \mapsto C_{p,q}^{\sigma}
\end{align*}
has image an $R$-basis of $A'$.
\end{definition}

If $G'(\lambda)=G(\lambda)$ for each $\lambda$, that is, if we have only restricted the set of available link states, then we will say that $A'$ is obtained from $A$ \emph{by restriction of link states}. Similarly, if $M'(\lambda)=M(\lambda)$ for each $\lambda$ then we will say that $A'$ is obtained from $A$ \emph{by restriction of groups}.

\begin{proposition} \label{prop:restriction} Let $A$ be a diagram-like algebra, with {\wossit} datum $(\Lambda,G,M,C,*)$. Let $A' \subset A$ be a subalgebra. If $A'$ is obtained by restriction in the sense of Definition \ref{def:restriction}, then $A'$ is a diagram-like algebra, with {\wossit} datum $(\Lambda\lvert_{M'},G',M',C,*)$, where $$\Lambda\lvert_{M'} = \{\lambda \in \Lambda \mid M'(\lambda) \neq \emptyset\},$$ and we abuse notation to identify $G',C$, and $*$ with appropriate restrictions. \end{proposition}

Note that we are taking it as given that $A'$ is closed under the multiplication in $A$, but not necessarily that it is closed under the involution $*$.

\begin{proof} Either $\Lambda\lvert_{M'}$ is empty, in which case $A'=0$ and we are done, or it is non-empty, and inherits a poset structure as a subset of $\Lambda$. We must verify Axioms \ref{item:W1} and \ref{item:W2} of Definition \ref{algDef}, as well as the dependency conditions of Definition \ref{def:dgrmLike}.

Axiom \ref{item:W1} follows immediately from the definition of restriction (Definition \ref{def:restriction}). Axiom \ref{item:W2} holds (in the sense that the involution $*$ preserves $A'$), since the basis $C_{p,q}^{\sigma}$ of $A'$ is closed under interchanging $p$ and $q$, and under inversion of $\sigma$ (since each $G'(\lambda)$ is a subgroup of $G(\lambda)$). We have assumed that $A'$ is multiplicatively closed, and the dependency conditions hold in $A'$ since they held already in $A$. This completes the proof.
\end{proof}

\section{Examples} \label{section:Ex}

We will use the `priming/unpriming' convention of \cite[Definition 3.1]{Boyde2}. The sets that we will use as vertex labels will consist of a positive integer, together with some number of superscript `prime' markings. In particular, we write $\underline{n}$ for the set of symbols $\{1,2,\dots , n\}$, and $\underline{n}'$ for the set $\{1',2', \dots, n'\}$.

\begin{example} \label{ex:Brauer} Let $R$ be a commutative ring with unit, and let $\delta \in R$. We take the \emph{Brauer algebra} $\Brauer_n(\delta)$ \cite{Brauer} to be defined as having an $R$-basis consisting of partitions of the set $\underline{n} \cup \underline{n}'$ where each part contains two elements. As is standard (see e.g. \cite{GrahamLehrer} or \cite{BHP}), we will think of these basis elements as diagrams on the vertex set $\underline{n} \cup \underline{n}'$, where an edge is drawn between pairs of vertices which lie in the same part of the partition. The multiplication is given by `concatenation of diagrams and replacement of loops by factors of $\delta$' (c.f. \cite{GrahamLehrer,BHP}). \end{example}

\begin{remark} \label{rmk:NotTheSame}Graham and Lehrer show \cite[Theorem 4.10]{GrahamLehrer} that the the Brauer algebra $\Brauer_n(\delta)$ is cellular. It then follows from Proposition \ref{prop:cellIsWossit} that $\Brauer_n(\delta)$ is {\wossit}. We will work with a more elementary {\wossit} structure. For the reader familiar with \cite{GrahamLehrer}, our basis elements $C_{p,q}^{\sigma}$ are what Graham and Lehrer would denote $[p,q,\sigma]$ (though they prefer to use different letters, as in $[S,T,w]$).
\end{remark}

\begin{definition} \label{def:BrauerDGL}
Formally, we assign a datum $(\Lambda,G,M,C,*)$ (which is independent of $\delta$, and will be seen to make $\Brauer_n(\delta)$ into a diagram-like algebra) as follows:
\begin{itemize}
    \item $\mathbf{\Lambda:}$ Let $\Lambda = \{t \in \underline{n}_0 \mid  n-t \textrm{ is even} \}$.
    \item $\mathbf{G:}$ For each $t \geq 1 \in \Lambda$, let $G(t)=\Sigma_t$ be the symmetric group of order $t$. We regard $\Sigma_0$ as the trivial group, i.e. as self-bijections of the empty set $\emptyset$.
    \item $\mathbf{M:}$ Let $M(t)$ be the \emph{Brauer link states on $\underline{n}$ with $t$ defects}: these are partitions of $\underline{n}$ into parts of cardinality either 1 or 2. The parts of cardinality 1 (which may be identified with their single element) will be called \emph{defects} or \emph{defect vertices}.
    \item $\mathbf{C:}$ Let $t \in \Lambda$. For $t \geq 0$, $\sigma \in \Sigma_t$, and $p,q \in M(t)$ the basis element $C_{p,q}^\sigma$ is the unique partition of $\underline{n} \cup \underline{n}'= \{1, \dots, n,1', \dots n'\}$ such that:
    \begin{itemize}
        \item The restriction of $C_{p,q}^\sigma$ to $\underline{n} = \{1, \dots , n\}$ is $p$.
        \item The restriction of $C_{p,q}^\sigma$ to $\underline{n}' = \{1', \dots n'\}$ is the partition $q'$ which corresponds to $q$ under the priming/unpriming bijection $\{1, \dots , n\} \cong \{1', \dots, n'\}$.
        \begin{itemize}
            \item If $t=0$, then ($G(t)$ is the trivial group, and) $C_{p,q}^1$ has no parts containing both an element of $\underline{n}$ and $\underline{n}'$.
        \item If $t \geq 1$, then the partition $C_{p,q}^\sigma$ will have $t$ parts containing both an element of of $\underline{n}$ and an element of $\underline{n}'$, as follows. The $t$ defects in $p$ are a well-ordered subset of  $\underline{n}$, and this defines a bijection with $\underline{t}$. The vertex mapped to $i \in \underline{t}$ by this correspondence will be called the $i$-th \emph{$i$-th defect}. Similarly, we put the parts of $q'$ having a defect in bijection with $\underline{t}'$. The partition $C_{p,q}^\sigma$ consists of the cardinality 2 parts of $p$ and $q'$, together with the unions of the $i'-th$ defect vertex of $q'$ with the $\sigma(i)$-th defect vertex of $p$ for each $i$.
        \end{itemize}
    \end{itemize}
    \item $\mathbf{*:}$ The anti-involution $*$ is the unique bijection which interchanges $\underline{n}$ and $\underline{n}'$, and preserves the order of each of these subsets individually.
\end{itemize}
\end{definition}

The following definition will be helpful.

\begin{definition} \label{def:pairGraph} Let $q,p$ be Brauer link states with $t_1$ and $t_2$ defects respectively. The \emph{pair graph} $\Gamma_{q,p}$ is the graph with vertices $\underline{n}$, an edge (thought of as `on the left') connecting each pair of vertices which are connected in $q$, and an edge (thought of as `on the right') connecting each pair of vertices which are connected in $p$. The \emph{pair set} $S_{\langle q,p \rangle}$ is the subset of the path components $\pi_0(\Gamma_{\langle q,p \rangle})$ given by the intersection of the maps $\underline{t}_1 \to \Gamma_{\langle q,p \rangle} \to \pi_0(\Gamma_{\langle q,p \rangle})$ and $\underline{t}_1 \to \pi_0(\Gamma_{\langle q,p \rangle})$ marking the defects. In other words, $S_{\langle q,p \rangle}$ is the set of path components in the pair graph which contain a defect vertex from $q$ and a defect vertex from $p$. \end{definition}

\begin{example} Taking $n=11$, if \begin{center}
$q =$
    \quad
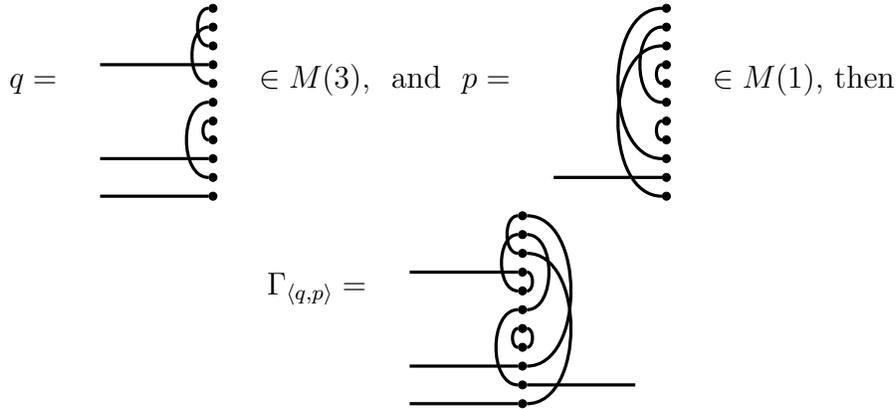
\begin{tikzpicture}[x=1.5cm,y=-.5cm,baseline=-1.05cm]

\def\wid{2}
\def\hei{0.5}
\def\nodesize{3}
\def\ang{90}

\node[v, minimum size=\nodesize] (c1) at (2* \wid,0*\hei) {};
\node[v, minimum size=\nodesize] (c2) at (2* \wid,1*\hei) {};
\node[v, minimum size=\nodesize] (c3) at (2* \wid,2*\hei) {};
\node[v, minimum size=\nodesize] (c4) at (2* \wid,3*\hei) {};
\node[v, minimum size=\nodesize] (c5) at (2* \wid,4*\hei) {};
\node[v, minimum size=\nodesize] (c6) at (2* \wid,5*\hei) {};
\node[v, minimum size=\nodesize] (c7) at (2* \wid,6*\hei) {};
\node[v, minimum size=\nodesize] (c8) at (2* \wid,7*\hei) {};
\node[v, minimum size=\nodesize] (c9) at (2* \wid,8*\hei) {};
\node[v, minimum size=\nodesize] (c10) at (2* \wid,9*\hei) {};
\node[v, minimum size=\nodesize] (c11) at (2* \wid,10*\hei) {};

\draw[e] (c1) to[out=180, in=180] (c3);
\draw[e] (c2) to[out=180, in=180] (c5);
\draw[e] (c7) to[out=180, in=180] (c8);
\draw[e] (c6) to[out=180, in=180] (c10);

\draw[e] (c4) to [out=180,in=0] (1.5*\wid,3*\hei);
\draw[e] (c9) to [out=180,in=0] (1.5*\wid,8*\hei);
\draw[e] (c11) to [out=180,in=0] (1.5*\wid,10*\hei);

\end{tikzpicture}
\quad $\in M(3),$ \textrm{ and } $p =$
    \quad
\begin{tikzpicture}[x=1.5cm,y=-.5cm,baseline=-1.05cm]

\def\wid{2}
\def\hei{0.5}
\def\nodesize{3}
\def\ang{90}

\node[v, minimum size=\nodesize] (c1) at (2* \wid,0*\hei) {};
\node[v, minimum size=\nodesize] (c2) at (2* \wid,1*\hei) {};
\node[v, minimum size=\nodesize] (c3) at (2* \wid,2*\hei) {};
\node[v, minimum size=\nodesize] (c4) at (2* \wid,3*\hei) {};
\node[v, minimum size=\nodesize] (c5) at (2* \wid,4*\hei) {};
\node[v, minimum size=\nodesize] (c6) at (2* \wid,5*\hei) {};
\node[v, minimum size=\nodesize] (c7) at (2* \wid,6*\hei) {};
\node[v, minimum size=\nodesize] (c8) at (2* \wid,7*\hei) {};
\node[v, minimum size=\nodesize] (c9) at (2* \wid,8*\hei) {};
\node[v, minimum size=\nodesize] (c10) at (2* \wid,9*\hei) {};
\node[v, minimum size=\nodesize] (c11) at (2* \wid,10*\hei) {};

\draw[e] (c2) to[out=180, in=180] (c6);
\draw[e] (c5) to[out=180, in=180] (c4);
\draw[e] (c7) to[out=180, in=180] (c8);
\draw[e] (c9) to[out=180, in=180] (c1);
\draw[e] (c3) to[out=180, in=180] (c11);

\draw[e] (c10) to [out=180,in=0] (1.5*\wid,9*\hei);

\end{tikzpicture}
\quad $\in M(1),$ then
\end{center}

\begin{center}
$\Gamma_{\langle q,p \rangle} =$
    \quad
\begin{tikzpicture}[x=1.5cm,y=-.5cm,baseline=-1.05cm]

\def\wid{2}
\def\hei{0.5}
\def\nodesize{3}
\def\ang{90}

\node[v, minimum size=\nodesize] (c1) at (2* \wid,0*\hei) {};
\node[v, minimum size=\nodesize] (c2) at (2* \wid,1*\hei) {};
\node[v, minimum size=\nodesize] (c3) at (2* \wid,2*\hei) {};
\node[v, minimum size=\nodesize] (c4) at (2* \wid,3*\hei) {};
\node[v, minimum size=\nodesize] (c5) at (2* \wid,4*\hei) {};
\node[v, minimum size=\nodesize] (c6) at (2* \wid,5*\hei) {};
\node[v, minimum size=\nodesize] (c7) at (2* \wid,6*\hei) {};
\node[v, minimum size=\nodesize] (c8) at (2* \wid,7*\hei) {};
\node[v, minimum size=\nodesize] (c9) at (2* \wid,8*\hei) {};
\node[v, minimum size=\nodesize] (c10) at (2* \wid,9*\hei) {};
\node[v, minimum size=\nodesize] (c11) at (2* \wid,10*\hei) {};

\draw[e] (c1) to[out=180, in=180] (c3);
\draw[e] (c2) to[out=180, in=180] (c5);
\draw[e] (c7) to[out=180, in=180] (c8);
\draw[e] (c6) to[out=180, in=180] (c10);

\draw[e] (c4) to [out=180,in=0] (1.5*\wid,3*\hei);
\draw[e] (c9) to [out=180,in=0] (1.5*\wid,8*\hei);
\draw[e] (c11) to [out=180,in=0] (1.5*\wid,10*\hei);

\draw[e] (c2) to[out=0, in=0] (c6);
\draw[e] (c5) to[out=0, in=0] (c4);
\draw[e] (c7) to[out=0, in=0] (c8);
\draw[e] (c9) to[out=0, in=0] (c1);
\draw[e] (c3) to[out=0, in=0] (c11);

\draw[e] (c10) to [out=0,in=180] (2.5*\wid,9*\hei);

\end{tikzpicture}
\quad
\end{center}
and $S_{\langle q,p \rangle} \subset \pi_0(\Gamma_{\langle q,p \rangle})$ consists of the single indicated (dashed) path component:
\begin{center}
\quad
\begin{tikzpicture}[x=1.5cm,y=-.5cm,baseline=-1.05cm]

\def\wid{2}
\def\hei{0.5}
\def\nodesize{3}
\def\ang{90}

\node[v, minimum size=\nodesize] (c1) at (2* \wid,0*\hei) {};
\node[v, minimum size=\nodesize] (c2) at (2* \wid,1*\hei) {};
\node[v, minimum size=\nodesize] (c3) at (2* \wid,2*\hei) {};
\node[v, minimum size=\nodesize] (c4) at (2* \wid,3*\hei) {};
\node[v, minimum size=\nodesize] (c5) at (2* \wid,4*\hei) {};
\node[v, minimum size=\nodesize] (c6) at (2* \wid,5*\hei) {};
\node[v, minimum size=\nodesize] (c7) at (2* \wid,6*\hei) {};
\node[v, minimum size=\nodesize] (c8) at (2* \wid,7*\hei) {};
\node[v, minimum size=\nodesize] (c9) at (2* \wid,8*\hei) {};
\node[v, minimum size=\nodesize] (c10) at (2* \wid,9*\hei) {};
\node[v, minimum size=\nodesize] (c11) at (2* \wid,10*\hei) {};

\draw[e] (c1) to[out=180, in=180] (c3);
\draw[e,dashed] (c2) to[out=180, in=180] (c5);
\draw[e] (c7) to[out=180, in=180] (c8);
\draw[e,dashed] (c6) to[out=180, in=180] (c10);

\draw[e,dashed] (c4) to [out=180,in=0] (1.5*\wid,3*\hei);
\draw[e] (c9) to [out=180,in=0] (1.5*\wid,8*\hei);
\draw[e] (c11) to [out=180,in=0] (1.5*\wid,10*\hei);

\draw[e,dashed] (c2) to[out=0, in=0] (c6);
\draw[e,dashed] (c5) to[out=0, in=0] (c4);
\draw[e] (c7) to[out=0, in=0] (c8);
\draw[e] (c9) to[out=0, in=0] (c1);
\draw[e] (c3) to[out=0, in=0] (c11);

\draw[e,dashed] (c10) to [out=0,in=180] (2.5*\wid,9*\hei);

\end{tikzpicture}
\quad
\end{center}
\end{example}

\begin{proposition} \label{prop:BrauerDGL} With the datum of Definition \ref{def:BrauerDGL}, the Brauer algebra $\Brauer_n(\delta)$ is diagram-like. In particular, given elements $C_{p_1,q_1}^{\sigma_1}$ associated to $t_1 \in \Lambda$ and $C_{p_2,q_2}^{\sigma_2}$ associated to $t_2 \in \Lambda$, we have
$$C_{p_1,q_1}^{\sigma_1} C_{p_2,q_2}^{\sigma_2} = \delta^i C_{p,q}^{\sigma},$$
where 
\begin{itemize}
    \item $i$ is the number of loops in the pair graph $\Gamma_{\langle q_1, p_2 \rangle}$, and
    \item the element $t \in \Lambda$ associated to the product $C_{p,q}^{\sigma}$ is equal to $t_2$ if and only if the cardinality of the pair-set $S_{\langle q_1, p_2 \rangle}$ is $t_2$.
\end{itemize}
 \end{proposition}

\begin{proof} The ingredients we need are all proven by Graham and Lehrer (in different notation). They show \cite[Proposition 4.4]{GrahamLehrer} that the $C_{p,q}^{\sigma}$ do indeed form an $R$-module basis of $\Brauer_n(\delta)$ (which is the canonical basis used to define $\Brauer_n(\delta)$) verifying Condition (\ref{item:W1}) of Definition \ref{algDef}.

The first bullet point above then follows immediately from the definition of the Brauer algebras. Likewise, Condition (\ref{item:W2}) follows immediately from the definitions, and the dependency conditions of Definition \ref{def:dgrmLike} are verified by Graham and Lehrer in different notation \cite[Proposition 4.7]{GrahamLehrer}. Precisely, it follows from the definition of the Brauer algebra that a product of two basis elements takes the form:
$$C_{p_1,q_1}^{\sigma_1} C_{p_2,q_2}^{\sigma_2} = \kappa C_{p,q}^{\sigma},$$
and
\begin{itemize}
    \item Conditions (\ref{item:dl1}) and (\ref{item:dl2}) follows from \cite[Remark 4.6iii]{GrahamLehrer}.
    \item Conditions (\ref{item:dl3}) and (\ref{item:dl4}) follow respectively from the third and second displayed equations in \cite[Proposition 4.7]{GrahamLehrer}. The second bullet point above follows likewise from the third equation.
    \item Condition (\ref{item:dl5}) follows from the same two displayed equations: specifically, if $\lambda = \lambda_2$, then (in Graham and Lehrer's notation) then the restriction functions become the identity, the second displayed equation becomes $$ S_2'' = S_2',$$ (so $q = q_2$), and the third displayed equation becomes $$w'' = w \lvert_{T_{S_2}(S_2,S_1')} w(S_2,S_1') w',$$ so $$w'' (w')^{-1} = w \lvert_{T_{S_2}(S_2,S_1')} w(S_2,S_1'),$$ so $\sigma \sigma_2^{-1} = w'' (w')^{-1}$ depends only on $q_1=S_2$, $p_2=S_1'$, and $\sigma_1=w$, as required.
\end{itemize}
This verifies the conditions of Definition \ref{def:dgrmLike}, and completes the proof. \end{proof}

Recall from e.g. \cite{Jones, GrahamLehrer} that the Jones annular algebra $\Jones_n(\delta)$ is the subalgebra of $\Brauer_n(\delta)$ spanned by diagrams `which can be embedded on the cylinder without overlapping edges'. We will freely conflate the vertex set $\underline{n}$ with the cyclic group $C_n$.

The following definition is due to Graham and Lehrer \cite[Lemma 6.12]{GrahamLehrer}.

\begin{definition} \label{def:AnnularLS} Let $p$ be a Brauer link state, that is to say, a partition of $\underline{n}$ into subsets of cardinality 1 or 2, and that the defect parts are precisely those of cardinality 1. We will say that $p$ is \emph{annular} if whenever $i$ and $j$ are connected by an edge in $p$, we have:
    \begin{itemize}
        \item No edge in $p$ connects a vertex from the cyclic interval $(i,j)$ to one from $(j,i)$. In terms of partitions, $(i,j)$ and $(j,i)$ are unions of parts of $p$.
        \item Either all defects of $p$ are contained in $(i,j)$, or all defects of $p$ are contained in $(j,i)$.
    \end{itemize}
    
\end{definition}

\begin{example} \label{ex:Jones} As in \cite{Boyde2}, by the \emph{cyclic interval} $[i,j]$ in the cyclic group $C_n$ we mean the set $\{i,i+1,i+2, \dots, j \}$. Open cyclic intervals are defined similarly. For example, we then have that $[i,j]$ and $(j,i)$ are disjoint, and that their union is $C_n$. In what follows, our presentation will differ from Graham and Lehrer's because we speak in cyclic intervals.

Graham and Lehrer \cite[Proposition 6.14]{GrahamLehrer} show (with different terminology) that annular partitions are precisely those $\rho = C_{p,q}^{\sigma}$ (corresponding to $t \in \Lambda$) for which:
\begin{itemize}
    \item $\sigma \in C_t \leq \Sigma_t$ is an element of the cyclic group of order $t$, and
    \item $p$ and $q$ are annular link states in the sense of Definition \ref{def:AnnularLS}.
\end{itemize}

These results of Graham and Lehrer show that $\Jones_n(\delta)$ is obtained from $\Brauer_n(\delta)$, by restriction (Definition \ref{def:restriction}) of both link states and groups.

It follows by Proposition \ref{prop:restriction} that $\Jones_n(\delta)$ is a diagram-like algebra. Its {\wossit} datum coincides with that of $\Brauer_n(\delta)$ except that:
\begin{itemize}
    \item For $t \in \Lambda  = \{t \in \underline{n}_0 \mid n-t \textrm{ is even} \}$, we set $G(t) = C_t$, the cyclic group of order $t$ (with the convention that $C_0$ is the trivial group).
    \item For $t \in \Lambda$, $M(t)$ is the set of annular link states with $t$ defects.
\end{itemize}
\end{example}

Again, the following definition is due to Graham and Lehrer \cite[Lemma 6.2]{GrahamLehrer}.

\begin{definition} \label{def:PlanarLS} We will say that a Brauer link state $p$ is \emph{planar} if whenever $i<j$ and $i,j$ are connected in $p$, we have:
\begin{itemize}
    \item If $k, \ell$ lie in the same part of $p$ and $k$ lies in the closed interval $[i,j]$, then $\ell$ also lies in $[i,j]$. In terms of partitions, $[i,j]$ is a union of parts of $p$, and
    \item the closed interval $[i,j]$ contains no defects.
\end{itemize}
\end{definition}

\begin{example} \label{ex:TL} We take the \emph{Temperley-Lieb algebra} $\TL_n(\delta)$ to be a subalgebra of $\Brauer_n(\delta)$, defined identically to the Jones annular algebra, except that diagrams are now required to have planar representatives in the square $[0,1] \times [0,1]$ when $i$ is embedded as $(\frac{i-1}{n-1},0)$ and $i'$ is embedded as $(\frac{i-1}{n-1},1)$. These algebras were originally defined by Temperley and Lieb \cite{TemperleyLieb}, and their diagrammatic interpretation is due to Kauffman \cite{Kauffman}.

Graham and Lehrer show \cite[Proposition 6.5ii]{GrahamLehrer} that the Temperley-Lieb algebra has a basis consisting of those partitions $\rho = C_{p,q}^{\sigma}$ for which $\sigma = \mathrm{id}$ is the identity permutation, and $p$ and $q$ are \emph{planar} link states.

As usual, this implies that $\TL_n(\delta)$ is diagram-like (obtained by restriction from $\Jones_n(\delta)$), with diagram-like data which coincides with that of $\Brauer_n(\delta)$ except that:
\begin{itemize}
    \item Each $G(t)$ is trivial, and
    \item $M(t)$ is the set of planar link states with $t$ defects.
\end{itemize}
\end{example}

\section{Basic properties} \label{section:BasicProperties}

\subsection{Twosided Ideals}

Throughout this subsection, to clean up the statements, $A$ will be a fixed {\wossit} algebra, with {\wossit} datum $(\Lambda, G, M, C, *)$.

\begin{definition} \label{def:I} Let $X \subset \Lambda$. Define
$$I_X := \Span_R\{C_{p,q}^{\sigma} \mid \mu \in X, \sigma \in g(\mu), p,q \in M(\mu)\}.$$
In particular, given $\lambda \in \Lambda$, write $$I_{\leq \lambda} := \Span_R\{C_{p,q}^{\sigma} \mid \mu \leq \lambda, \sigma \in g(\mu), p,q \in M(\mu)\},$$
and
$$I_{< \lambda} := \Span_R\{C_{p,q}^{\sigma} \mid \mu < \lambda, \sigma \in g(\mu), p,q \in M(\mu)\}.$$
\end{definition}

The following lemma is an immediate consequence of Axiom \ref{item:W2}.

\begin{lemma} \label{lem:IStar} Let $X \subset \Lambda$ be any subset. Then $(I_X)^*=I_X$. \qed
\end{lemma}

A free $R$-module modulo a subset of a basis is again a free $R$-module on the complement of this subset. In our case, we have:

\begin{lemma} \label{lem:quotientBasis} Let $X \subset \Lambda$ be any subset, and let $\pi_X : A \to \faktor{A}{I_X}$ denote the projection. Then $\faktor{A}{I_X}$ is a free $R$-module with basis \[\pushQED{\qed} 
\{\pi_X(C_{p,q}^{\sigma}) \mid \mu \not\in X, \sigma \in g(\mu), p,q \in M(\mu)\}. \qedhere
\popQED
\] \end{lemma}

Recall that a subset $X$ of a poset $\Lambda$ is said to be downward closed if whenever $\mu \leq \lambda$ and $\lambda \in X$ then $\mu \in X$.

\begin{lemma} \label{lem:IIdeal} Let $X \subset \Lambda$ be downward closed. Then $I_X$ is a twosided ideal of $A$. In particular, for each $\lambda \in \Lambda$, both $I_{\leq \lambda}$ and $I_{< \lambda}$ are twosided ideals.
\end{lemma}

\begin{proof} Since $X \subset \Lambda$ is downward closed, we have that if $\mu \in X$ then $I_{< \mu} \subset I_X$. The result then follows from Axiom \ref{item:W3} and its dual, Remark \ref{rmk:W3}. \end{proof}

If $I_X$ is a twosided ideal of $A$, then the quotient $\faktor{A}{I_X}$ is again an $R$-algebra. In fact, it is again a {\wossit} algebra:

\begin{proposition} \label{prop:quotientWossit} Let $X \subset \Lambda$ be a downward closed proper subset. Then $\faktor{A}{I_X}$ is a {\wossit} algebra with {\wossit} datum $$(\Lambda \setminus X, G, M, \pi_X \circ C, *),$$ where $\pi_X$ is the projection $A \to \faktor{A}{I_X}$, and we abuse notation by still writing $G$, $M$, and $C$ for the restrictions of those assignments to $\Lambda \setminus X$, and by identifying $*$ with the induced map on the quotient.
\end{proposition}

\begin{proof} Since $X$ is downward closed, $I_X$ is a twosided ideal, so $\faktor{A}{I_X}$ is again an associative $R$-algebra. Since $X$ is proper, $I_X$ is a proper ideal of $A$, so $\faktor{A}{I_X}$ is unital. The set $\Lambda \setminus X$ is a poset by restricting the order. The functions $G$ and $M$ still assign a group and a set respectively to each $\lambda \in \Lambda \setminus X$, and $\pi_X \circ C$ assigns an element of $\faktor{A}{I_X}$, to each triple $(p, \sigma, q)$ in $ \coprod_{\lambda \in \Lambda \setminus X} M(\lambda) \times G(\lambda) \times M(\lambda)$. The involution $*$ descends to the quotient by Lemma \ref{lem:IStar}. This establises that the datum is of the correct form, and we must now check that it satisfies the axioms.

Axiom \ref{item:W1} follows from Lemma \ref{lem:quotientBasis}. Axiom \ref{item:W2} is automatic. Axiom \ref{item:W3} is essentially the correspondence theorem: note that for $\lambda \in \Lambda \setminus X$, the projection $\pi_X$ gives an isomorphism between the ideal $\pi_X(I_{< \lambda})$ of $\faktor{A}{I_X}$ and the ideal $I_{< \lambda} + I_X$ of $A$. This latter ideal certainly contains $I_{< \lambda}$, so Axiom \ref{item:W3} for $\faktor{A}{I_X}$ is implied by Axiom \ref{item:W3} for $A$. \end{proof}

\begin{remark} \label{rmk:basisAbuse} We will often abuse notation by writing $C_{p,q}^{\sigma}$ for $\pi_X(C_{p,q}^{\sigma})$, which is to say we will identify the basis of $\faktor{A}{I_X}$ with the appropriate subset of the basis of $A$.
\end{remark}

If $\lambda$ is a minimal element of $\Lambda \setminus X$, then $I_{< \lambda} \subset I_X$. This implies the following refinement of Axiom \ref{item:W3}, and its dual, Remark \ref{rmk:W3}, in the quotient $\faktor{A}{I_X}$.

\begin{proposition} \label{prop:quotientAx3} Let $X$ be a downward closed subset of $\Lambda$, and consider the quotient algebra $\faktor{A}{I_X}$. For any minimal element $\lambda$ of $\Lambda \setminus X$, any $\sigma \in G(\lambda)$, any $p,q \in M(\lambda)$, and any $a \in \faktor{A}{I_X}$, we have the equalities $$a C_{p,q}^{\sigma} = \sum_{\substack{\sigma' \in G(\lambda) \\ p' \in M(\lambda) }} r_a(p',\sigma'\sigma^{-1},p) C_{p',q}^{\sigma'}$$ and $$C_{p,q}^{\sigma}a = \sum_{\substack{\sigma' \in G(\lambda) \\ q' \in M(\lambda)}} r_{a^*}(q',(\sigma')^{-1}\sigma,q)C_{p,q'}^{\sigma'}$$ in $\faktor{A}{I_{X}}$. \qed
\end{proposition}

The following two lemmas are elementary.

\begin{lemma} \label{lem:XcupLambda} If $X$ is a downward closed subset of a poset $\Lambda$, and $\lambda$ is a minimal element of $\Lambda \setminus X$, then $X \cup \{ \lambda \}$ is again downward closed. \qed
\end{lemma}

\begin{lemma} \label{lem:YcupLambda} If $Y \subset X$ are downward closed subsets of a poset $\Lambda$, and $\lambda$ is minimal in $X \setminus Y$, then $Y \cup \{\lambda\}$ is again downward closed. \qed
\end{lemma}

\subsection{The product of two basis elements}

For any minimal element $\lambda$ of $\Lambda \setminus X$, consider the product of two basis elements $C_{p_1,q_1}^{\sigma_1} C_{p_2,q_2}^{\sigma_2}$ using the formulae of Proposition \ref{prop:quotientAx3} (with $\sigma_1, \sigma_2 \in G(\lambda)$, and any $p_1,p_2,q_1,q_2 \in M(\lambda)$). The first formula shows that only basis elements of the form $C_{p',q_2}^{\sigma'}$ may appear with nonzero coefficient, and the second shows that only basis elements of the form $C_{p_1,q'}^{\sigma'}$ may appear with nonzero coefficient, so in fact only basis elements of the form $C_{p_1,q_2}^{\sigma'}$ may appear.

We therefore have: \begin{align*} C_{p_1,q_1}^{\sigma_1}  C_{p_2,q_2}^{\sigma_2} & = \sum_{\substack{\sigma' \in G(\lambda)}} r_{C_{p_1,q_1}^{\sigma_1}}(p_1,\sigma'\sigma_2^{-1},p_2) C_{p_1,q_2}^{\sigma'} \\ & = \sum_{\substack{\sigma' \in G(\lambda)}} r_{C_{q_2,p_2}^{\sigma_2^{-1}}}(q_2,(\sigma')^{-1}\sigma_1,q_1)C_{p_1,q_2}^{\sigma'}
\end{align*}

Equating coefficients, we have $$r_{C_{p_1,q_1}^{\sigma_1}}(p_1,\sigma'\sigma_2^{-1},p_2) = r_{C_{q_2,p_2}^{\sigma_2^{-1}}}(q_2,(\sigma')^{-1}\sigma_1,q_1)$$ for each $\sigma' \in G(\lambda)$. In this equality, $p_1$ appears only on the left hand side, $q_2$ appears only on the left, and the two appear in the same two entries. Thus, $r_{C_{p_1,q_1}^{\sigma_1}}(p_1,\sigma'\sigma_2^{-1},p_2)$ is independent of $p_1$, and the function $$s(\sigma_1,q_1,p_2,\tau) := r_{C_{p_1,q_1}^{\sigma_1}}(p_1,\tau,p_2) \in R$$ is well defined. We have established the following corollary of Proposition \ref{prop:quotientAx3}.

\begin{corollary} \label{cor:basisMultFormula} Let $X$ be a downward closed subset of $\Lambda$, and consider the quotient algebra $\faktor{A}{I_X}$. For any minimal element $\lambda$ of $\Lambda \setminus X$, any $\sigma_1, \sigma_2 \in G(\lambda)$, and any $p_1,p_2,q_1,q_2 \in M(\lambda)$, we have 
\begin{align*} C_{p_1,q_1}^{\sigma_1}  C_{p_2,q_2}^{\sigma_2} & = \sum_{\substack{\sigma' \in G(\lambda)}} s(\sigma_1,q_1,p_2,\sigma'\sigma_2^{-1}) C_{p_1,q_2}^{\sigma'} \\ & = \sum_{\substack{\sigma' \in G(\lambda)}} s(\sigma_2^{-1},p_2,q_1,(\sigma')^{-1}\sigma_1) C_{p_1,q_2}^{\sigma'}
\end{align*}
 in $\faktor{A}{I_{X}}$. In particular, for each $\sigma' \in G(\lambda)$ we have
\[\pushQED{\qed} 
s(\sigma_1,q_1,p_2,\sigma'\sigma_2^{-1})=s(\sigma_2^{-1},p_2,q_1,(\sigma')^{-1}\sigma_1). \qedhere
\popQED
\]
\end{corollary}

\subsection{Onesided ideals}

Again, in this subsection, $A$ will be a fixed {\wossit} algebra, with {\wossit} datum $(\Lambda, G, M, C, *)$.

\begin{definition} \label{def:J} Let $\lambda \in \Lambda$, and let $q \in M(\lambda)$. Define
$$J_q := \Span_R\{C_{p,q}^{\sigma} \mid \sigma \in g(\lambda), p \in M(\lambda)\} \subset A.$$
\end{definition}

If $\lambda$ is minimal in $\Lambda \setminus X$ then $I_{<\lambda} \subset I_X$. Axiom \ref{item:W3} then implies the following lemma.

\begin{lemma} \label{lem:JIdeal} If $X$ is a downward closed subset of $\Lambda$ (so that by Proposition \ref{prop:quotientWossit}, $\faktor{A}{I_X}$ is again a {\wossit} algebra), and $\lambda$ is minimal in $\Lambda \setminus X$, then $\pi_X(J_q)$ is a left ideal of $\faktor{A}{I_X}$. \qed
\end{lemma}


Axiom \ref{item:W1} says that if $q \neq q'$ then the canonical bases of $J_q$ and $J_{q'}$ are disjoint subsets of the canonical basis of $A$. This gives the following structural property.

\begin{lemma} \label{lem:JDisjt} Let $\lambda, \lambda' \in \Lambda$, $q \in M(\lambda)$, and $q' \in M(\lambda')$. If $q \neq q'$ then $J_q \cap J_{q'}=0$ in $A$. \qed \label{JInt}
\end{lemma}

From this lemma, we obtain the next one, which will be very important to us.

\begin{lemma} \label{lem:decomposition} If $X$ is a downward closed subset of $\Lambda$, and $\lambda$ is a minimal element of $\Lambda \setminus X$, then the inclusions $J_q \to I_\lambda$ induce an isomorphism $$\pi_X(I_{X \cup \{\lambda\}}) = \faktor{I_{X \cup \{\lambda\}}}{I_{X}} \cong \bigoplus_{q \in M(\lambda)} \pi_X(J_q)$$ of left $A$-modules.
\end{lemma}

\begin{proof} By Lemma \ref{lem:JIdeal}, the quotients $\pi_X(J_q)$ are indeed left $A$-modules. By Lemma \ref{lem:XcupLambda}, $X \cup \{ \lambda \}$ is again downwards closed, so by Lemma \ref{lem:IIdeal}, $I_{X \cup \{\lambda\}}$, hence the quotient, is a left (in fact twosided) $A$-module.

The twosided ideal $I_X$ of $A$ is defined to be the $R$-span of the elements $C_{p,q}^{\mu}$, for $\mu \in X$ and $p,q \in M(\mu)$. It follows that $\faktor{I_{X \cup \{\lambda\}}}{I_{X}}$ is the $R$-span of the elements $C_{p,q}^{\lambda}$ for $p,q \in M(\lambda)$, hence that every element of $\faktor{I_{X \cup \{\lambda\}}}{I_{X}}$ is an $R$-linear combination of elements drawn from the $J_q$, for $q \in M(\lambda)$: $$\faktor{I_{X \cup \{\lambda\}}}{I_{X}} \cong \sum_{q \in M(\lambda)} \pi_X(J_q).$$
By Lemma \ref{lem:JDisjt}, this sum is in fact direct. This completes the proof. \end{proof}

\subsection{Equivariance}

Let $R$ be a commutative ring, and let $A$ be a {\wossit} algebra over $R$, with {\wossit} datum $(\Lambda, G, M, C, *)$.

We here give a perhaps more intuitive reformulation of the equivariance condition of Axiom \ref{item:W3}, and its dual, Remark \ref{rmk:W3}. Let $\lambda \in \Lambda$, $\sigma, \tau \in G(\lambda)$ and $p,q \in M(\lambda)$ then for any element $a \in A$, Axiom \ref{item:W3} gives $$a C_{p,q}^{\sigma \tau} \equiv \sum_{\substack{\sigma'' \in G(\lambda) \\ p' \in M(\lambda) }} r_a(p',\sigma''\tau^{-1}\sigma^{-1},p) C_{p',q}^{\sigma''} (\mathrm{mod} I_{< \lambda}),$$ and reindexing the sum via $\sigma'=\sigma'' \tau^{-1}$ gives \begin{align*}
a C_{p,q}^{\sigma \tau}  &  \equiv \sum_{\substack{\sigma' \in G(\lambda) \\ p' \in M(\lambda) }} r_a(p',(\sigma' \tau) \tau^{-1}\sigma^{-1},p) C_{p',q}^{\sigma' \tau} (\mathrm{mod} I_{< \lambda}) \\
  & = \sum_{\substack{\sigma' \in G(\lambda) \\ p' \in M(\lambda) }} r_a(p',\sigma' \sigma^{-1},p) C_{p',q}^{\sigma' \tau} (\mathrm{mod} I_{< \lambda}).
\end{align*}

We record this congruence, and the dual obtained from Remark \ref{rmk:W3}, as the next lemma.

\begin{lemma} \label{lem:equivariance} For any $\lambda \in \Lambda$, $\sigma, \tau \in G(\lambda)$ and $p,q \in M(\lambda)$, and any $a \in A$, we have
$$a C_{p,q}^{\sigma \tau}  \equiv \sum_{\substack{\sigma' \in G(\lambda) \\ p' \in M(\lambda) }} r_a(p',\sigma' \sigma^{-1},p) C_{p',q}^{\sigma' \tau} (\mathrm{mod} I_{< \lambda}),$$ and \[\pushQED{\qed} 
C_{p,q}^{\tau\sigma}a \equiv \sum_{\substack{\sigma' \in G(\lambda) \\ q' \in M(\lambda)}} r_{a^*}(q',(\sigma')^{-1}\sigma,q)C_{p,q'}^{\tau\sigma'} (\mathrm{mod} I_{< \lambda}). \qedhere
\popQED
\]
\end{lemma}

For a product of basis elements corresponding to the same $\lambda$, in terms of the function $s$ of Corollary \ref{cor:basisMultFormula}, we get:

\begin{corollary} \label{cor:basisMultEquivariance} Let $X$ be a downward closed subset of $\Lambda$, and consider the quotient algebra $\faktor{A}{I_X}$. For any minimal element $\lambda$ of $\Lambda \setminus X$, any $\sigma_1, \sigma_2 \in G(\lambda)$, and any $p_1,p_2,q_1,q_2 \in M(\lambda)$, we have 
\begin{align*} C_{p_1,q_1}^{\tau_1 \sigma_1}  C_{p_2,q_2}^{\sigma_2 \tau_2} & = \sum_{\substack{\sigma' \in G(\lambda)}} s(\sigma_1,q_1,p_2,\sigma'\sigma_2^{-1}) C_{p_1,q_2}^{\tau_1 \sigma' \tau_2} \\ & = \sum_{\substack{\sigma' \in G(\lambda)}} s(\sigma_2^{-1},p_2,q_1,(\sigma')^{-1}\sigma_1) C_{p_1,q_2}^{\tau_1 \sigma' \tau_2}
\end{align*}
 in $\faktor{A}{I_{X}}$. \qed
\end{corollary}

One useful consequence of equivariance is the following, which will be necessary to show that the \emph{link modules} of the next section are indeed $A$-modules (Proposition \ref{prop:linkIsModule}).

\begin{lemma} \label{lem:eqProductReindexing}For each $\lambda \in \Lambda$, and every $a,b \in A$, the function $r_a$ satisfies the equation
$$\sum_{\sigma'\in G(\lambda)} r_{ab}(p',\sigma',p) = \sum_{\sigma'\in G(\lambda)} \sum_{\substack{\sigma'' \in G(\lambda) \\ p'' \in M(\lambda)}} r_a(p',\sigma',p'') r_b(p'',\sigma'',p)$$
\end{lemma}

\begin{proof} This lemma is essentially a consequence of the following: since $A$ is associative, we have
$$(ab)C_{p,q}^1 = a(bC_{p,q}^1).$$
Applying Axiom \ref{item:W3} to the left hand side, one sees that the coefficient of $C_{p',q}^{\rho}$ in $(ab)C_{p,q}^1$ is $r_{ab}(p',\rho,p)$. Applying the same axiom (twice) to the right hand side, one sees that the coefficient of the same basis element in $a(bC_{p,q}^1)$ is $\sum_{\sigma'', p''} r_a(p',\rho(\sigma'')^{-1},p'')r_b(p'',\sigma'',p)$. These coefficients must be equal, that is: $$r_{ab}(p',\rho,p) = \sum_{\sigma'', p''} r_a(p',\rho(\sigma'')^{-1},p'')r_b(p'',\sigma'',p).$$

The point is now that equivariance allows us to reindex in the desired manner on the right hand side, at least after an additional sum over $\rho$. Namely, \begin{align*} \sum_{\rho \in G(\lambda)} r_{ab}(p',\rho,p) & = \sum_{\rho \in G(\lambda)}\sum_{\substack{\sigma'' \in G(\lambda) \\ p'' \in M(\lambda)}} r_a(p',\rho(\sigma'')^{-1},p'')r_b(p'',\sigma'',p) \\
& = \sum_{\substack{(\rho,\sigma'') \in G(\lambda) \times G(\lambda) \\ p'' \in M(\lambda)}} r_a(p',\rho(\sigma'')^{-1},p'')r_b(p'',\sigma'',p) \\
& = \sum_{\substack{(\tau,\sigma'') \in G(\lambda) \times G(\lambda) \\ p'' \in M(\lambda)}} r_a(p',\tau,p'')r_b(p'',\sigma'',p) \\
& = \sum_{\tau \in G(\lambda)}\sum_{\substack{\sigma'' \in G(\lambda) \\ p'' \in M(\lambda)}} r_a(p',\tau,p'')r_b(p'',\sigma'',p),
\end{align*} where we used the fact that $$(\rho,\sigma'') \mapsto (\rho (\sigma'')^{-1}, \sigma'') = (\tau, \sigma'')$$ is a self-bijection of $G(\lambda) \times G(\lambda)$. This completes the proof. \end{proof}

\section{{\Link} modules and bilinear forms} \label{section:Form}

In this subsection, fix a commutative ring $R$, and a {\wossit} algebra $A$ over $R$, with {\wossit} datum $(\Lambda, G, M, C, *)$.

\begin{definition} \label{def:LinkMod} For $\lambda \in \Lambda$, let the \emph{{\link} module} $W(\lambda)$ be the free $R$-module on the symbols $C_p$, together with a left action of $A$ given by $$a C_p = \sum_{\substack{\sigma'\in G(\lambda) \\ p' \in M(\lambda)}} r_a(p',\sigma',p) C_{p'}.$$
\end{definition}

\begin{remark} \label{rmk:cellMod} If $A$ is a cellular algebra in the sense that each $G(\lambda)$ is trivial (Proposition \ref{prop:cellIsWossit}), then this definition reduces to Graham and Lehrer's definition \cite[Definition 2.1]{GrahamLehrer} of the \emph{cell module} associated to $\lambda$ (which they also denote by $W(\lambda)$). \end{remark}

\begin{remark} \label{rmk:mimicMult} As in the cellular case, note that the definition `mimics the multiplication from $A$' in the sense that Axiom \ref{item:W3} gives that, modulo $A(< \lambda)$, we have $$a C_{p,q}^{\tau} = \sum_{\substack{\sigma'\in G(\lambda) \\ p' \in M(\lambda)}} r_a(p',\sigma',p) C_{p',q}^{\sigma' \tau},$$ where we use the form of equivariance from Lemma \ref{lem:equivariance}. \end{remark}

We must verify that this multiplication rule actually makes $W(\lambda)$ into a left $A$-module.

\begin{proposition} \label{prop:linkIsModule} Each {\link} module $W(\lambda)$ is a left $A$-module. \end{proposition}

\begin{proof} We must check that the multiplication is bilinear, and that for $x \in W(\lambda)$ and $a,b\in A$, we have $1 \cdot x = x$, and $a(bx)=(ab)x$.

Since the algebra multiplication in $A$ is bilinear, for fixed $p',\sigma,$ and $p$, we have that $r_a(p',\sigma,p)$ is $R$-linear in $a$. The action of $A$ on $C_p$ is therefore $R$-linear in the first variable. Linearity in the second variable is automatic, because the action is defined as the $R$-linear extension of an action on the basis $C_p$. This establishes bilinearity, and it therefore suffices to check the equalities $1 \cdot x = x$ and $a(bx)=(ab)x$ in the case that $x = C_p$ is a basis element.

The first equality then follows immediately from Remark \ref{rmk:mimicMult}. For the second equality, we must check that $$\sum_{\substack{\sigma'\in G(\lambda) \\ p'\in M(\lambda)}} r_{ab}(p',\sigma',p) = \sum_{\substack{\sigma'\in G(\lambda)  \\ p'\in M(\lambda)}} \sum_{\substack{\sigma'' \in G(\lambda) \\ p'' \in M(\lambda)}} r_a(p',\sigma',p'') r_b(p'',\sigma'',p),$$ but this is just the sum over $p'$ of the equality given in Lemma \ref{lem:eqProductReindexing}. This completes the proof.
\end{proof}

\subsection{Bilinear forms}

Let $C_{p_1,q_1}^{\sigma_1}$ and $C_{p_2,q_2}^{\sigma_2}$ be basis elements associated to the same $\lambda \in \Lambda$. Recall that by Corollary \ref{cor:basisMultFormula}, there is a function $s$, defined by $$s(\sigma_1,q_1,p_2,\rho) := r_{C_{p_1,q_1}^{\sigma_1}}(p_1,\rho,p_2) \in R,$$ such that, modulo $I(< \lambda)$, we have 
\begin{align*} C_{p_1,q_1}^{\sigma_1}  C_{p_2,q_2}^{\sigma_2} & = \sum_{\substack{\sigma' \in G(\lambda)}} s(\sigma_1,q_1,p_2,\sigma'\sigma_2^{-1}) C_{p_1,q_2}^{\sigma'} \\ & = \sum_{\substack{\sigma' \in G(\lambda)}} s(\sigma_2^{-1},p_2,q_1,(\sigma')^{-1}\sigma_1) C_{p_1,q_2}^{\sigma'}
\end{align*}

\begin{definition} \label{def:innerProduct} Let $\lambda \in \Lambda$, and let $\tau \in M(\lambda)$. The \emph{bilinear form (on $W(\lambda)$) associated to $\tau$} is defined on the basis by setting $$\langle C_q,C_p \rangle_{\tau} : = s(1,q,p,\tau),$$ and then extended bilinearly over all of $W(\lambda) \times W(\lambda)$.
\end{definition}

In other words, $\langle C_q,C_p \rangle_{\tau}$ is the coefficient of $C_{p_1,q_2}^{\tau}$ in $C_{p_1,q}^{1}  C_{p,q_2}^{1}$ for any $p_1,q_2 \in M(\lambda)$. 

\begin{remark} \label{rmk:cellForm} If $A$ is a cellular algebra in the sense that each $G(\lambda)$ is trivial (Proposition \ref{prop:cellIsWossit}), then (Remark \ref{rmk:cellMod}) the module $W(\lambda)$ coincides with Graham and Lehrer's cell module, and this definition reduces immediately to Graham and Lehrer's definition \cite[Definition 2.3]{GrahamLehrer} of the bilinear form $\phi_\lambda$.
\end{remark}

We record some easy properties of this bilinear form.

\begin{proposition} \label{prop:bilFormProperties} \begin{enumerate} We have the following properties.
    \item $\langle C_q,C_p \rangle_{\tau}$ is equal to the coefficient of $C_{p_1,q_2}^{\sigma_1\tau\sigma_2}$ in $C_{p_1,q}^{\sigma_1}  C_{p,q_2}^{\sigma_2}$ for any $\sigma_1,\sigma_2 \in G(\lambda)$ and any $p_1,q_2 \in M(\lambda)$.
    \item We have the symmetry property $$\langle u,v \rangle_{\tau} = \langle v,u \rangle_{\tau^{-1}}.$$
\end{enumerate}
\end{proposition}

\begin{proof} Since $\langle C_q,C_p \rangle_{\tau}$ is the coefficient of $C_{p_1,q_2}^{\tau}$ in $C_{p_1,q}^{1}  C_{p,q_2}^{1}$ for any $p_1,q_2 \in M(\lambda)$, the first property follows from Corollary \ref{cor:basisMultEquivariance}.

By bilinearity, it suffices to verify the second property in the case $u=C_q, v=C_p$, but in this case the result is immediate from Corollary \ref{cor:basisMultFormula}. This completes the proof.
\end{proof}

If $A$ is a diagram-like algebra, then the definition of the bilinear form becomes more comprehensible. It is given by the next lemma, which follows immediately from the definitions.

\begin{lemma} \label{lem:dgrmLikeBilinearForm} Suppose that $A$ is diagram-like, let $\lambda \in \Lambda$, $p,q \in M(\lambda)$, and $\tau \in G(\lambda)$. From Definition \ref{def:dgrmLike}, we have, for any $p_1,q_2 \in G(\lambda)$, $$ C_{p_1,q}^1 C_{p,q_2}^1 = \kappa C_{p_1,q_2}^\sigma,$$ where the basis element on the right is associated to $\mu \in \Lambda$ depending only on $p$ and $q$, and $\kappa$ and $\sigma$ also depend only on $p,q$. Writing these two variables as functions of $p$ and $q$, we have \[\pushQED{\qed} 
\langle C_q,C_p \rangle_{\tau} = \begin{cases} \kappa(p,q) \chi_{\sigma(p,q)}(\tau) & \textrm{ if $\lambda = \mu(p,q)$, and } \\ 0 \textrm{ otherwise.}
\end{cases} \qedhere
\popQED
\] \end{lemma}

\begin{remark} If one just wishes to work with diagram algebras, it may be better to define a single bilinear form on $W(\lambda)$ as $\langle C_q, C_p\rangle = \kappa(q,p)$. We do not do so because this definition doesn't work well for naive-cellular algebras, and we wanted to be able to compare our Theorem \ref{thm:global} with Graham and Lehrer's results on cellular algebras (as in Section \ref{subsection:Cellular}). \end{remark}

The reason for introducing these bilinear forms is the next proposition. Recall the definition of the ideals $J_q$ (Definition \ref{def:J}), and recall that by Lemma \ref{lem:JIdeal}, if $X$ is a downward closed subset of $\Lambda$, and $\lambda$ is a minimal element of $\Lambda \setminus X$, then the image $\pi_X(J_q)$ in $\faktor{A}{I_X}$ is a left ideal.

\begin{proposition} \label{prop:formBegetIdempt} Let $X$ be a downward closed subset of $\Lambda$, and let $\lambda$ be a minimal element of $\Lambda \setminus X$. Fix $q \in M(\lambda)$.
Then there exists an idempotent $e_q$ in $\faktor{A}{I_X}$ which generates $\pi_X(J_q)$ as a left ideal if and only if the indicator function of the identity in $G(\lambda)$ lies in the $R$-span of the functions \begin{align*} G(\lambda) & \to R \\
\tau & \mapsto \langle C_q , C_p \rangle_{\tau \sigma^{-1}},
\end{align*} for $p \in M(\lambda)$, $\sigma \in G(\lambda)$.
\end{proposition}

We will use the following lemma in the proof.

\begin{lemma} \label{lem:idempotentGenCharacterisation} Let $A$ be an $R$-algebra, let $J$ be a left ideal of $A$, and let $e \in J$. Then $y e = y$ for every $y \in J$ if and only if $e$ is idempotent and generates $J$ as a left ideal.
\end{lemma}

\begin{proof} If $e$ is idempotent and generates $J$, then every element $y \in J$ has the form $x e$, so $y e = x e^2 = xe = y$.

Conversely, if $ye = y$ for every $y \in J$, then $e$ generates $J$ as a left ideal, and, since we have in particular that $e \in J$, setting $y = e$ gives that $e$ is idempotent. \end{proof}

\begin{proof}[Proof of Proposition \ref{prop:formBegetIdempt}] In this proof, we will commit the abuse of Remark \ref{rmk:basisAbuse}, and write $C_{p,q}^\sigma$ for the image of that basis vector in $\faktor{A}{I_X}$.

By Lemma \ref{lem:idempotentGenCharacterisation}, an element $e_q$ of $\pi_X(J_q)$ is an idempotent generator of $\pi_X(J_q)$ if and only if $y e_q = y$ for all $y \in \pi_X(J_q)$, if and only if $C_{p,q}^\sigma e_q = C_{p,q}^\sigma$ for every basis element $C_{p,q}^\sigma$ of $\pi_X(J_q)$.

Expressing $e_q$ in this basis gives $e_q = \sum_{\substack{p' \in M(\lambda) \\ \sigma' \in G(\lambda)}} \alpha_{p',\sigma'} C_{p',q}^{\sigma'}$, and we write:
\begin{align*}
    C_{p,q}^\sigma e_q & = C_{p,q}^\sigma \sum_{p', \sigma'} \alpha_{p',\sigma'} C_{p',q}^{\sigma'} \\
    & = \sum_{p', \sigma'} \alpha_{p',\sigma'} C_{p,q}^\sigma C_{p',q}^{\sigma'} \\
    & = \sum_{p', \sigma'} \alpha_{p',\sigma'} \sum_{\tau' \in G(\lambda)} \langle C_q, C_{p'}\rangle_{\tau'} C_{p,q}^{\sigma \tau' \sigma'} \\
    & =  \sum_{\tau'} \sum_{p', \sigma'} \alpha_{p',\sigma'} \langle C_q,  C_{p'}\rangle_{\tau'} C_{p,q}^{\sigma \tau' \sigma'} \\
    & =  \sum_{\tau}( \sum_{p', \sigma'} \alpha_{p',\sigma'} \langle C_q,  C_{p'}\rangle_{\tau (\sigma')^{-1}}) C_{p,q}^{\sigma \tau},
\end{align*}
where the third equality is by Proposition \ref{prop:bilFormProperties}. This last expression is equal to $C_{p,q}^\sigma$ if and only if $\sum_{p', \sigma'} \alpha_{p',\sigma'} \langle C_q,  C_{p'}\rangle_{\tau (\sigma')^{-1}}$ is equal to $1$ if $\tau$ is the identity, and zero otherwise.

That is, the coefficients $\alpha_{p',\sigma'}$ express the indicator function of the identity in $G(\lambda)$ as a linear combination of the $\tau \mapsto \langle C_q,  C_{p'}\rangle_{\tau (\sigma')^{-1}}$ if and only if they define an element $e_q$ (as a linear combination of $C_{p',q}^{\sigma'}$) which satisfies $C_{p,q}^\sigma e_q = C_{p,q}^\sigma$ for each $p$ and $\sigma$.

In particular, the indicator function can be expressed as such a linear combination if and only if $\pi_X(J_q)$ is principal and generated by an idempotent, as required. \end{proof}

We will use Proposition \ref{prop:formBegetIdempt} in conjunction with the following lemma. Note that this actually verifies something stronger than the hypothesis of that proposition, since it expresses the indicator function as a span of only those functions corresponding to a particular element $\sigma$.

\begin{lemma} \label{lem:ezCase} Let $\lambda \in \Lambda$, and fix $q \in M(\lambda)$. If there exists $v \in W(\lambda)$ such that for some $\sigma \in G(\lambda)$ we have $$\langle C_q, v \rangle_{\tau} = \begin{cases} 1 & \textrm{ if } \tau=\sigma, \textrm{ and} \\ 0 & \textrm{ otherwise,} \end{cases}$$ then the indicator function of $1 \in G(\lambda)$ lies in the span of the functions \begin{align*} G(\lambda) & \to R \\
\tau & \mapsto \langle C_q , C_p \rangle_{\tau \sigma},
\end{align*} for $p \in M(\lambda)$.
\end{lemma}

\begin{proof} The hypothesis gives $\chi_{\sigma}(\tau) = \langle C_q, v \rangle_{\tau}$.

Expressing $v$ in the basis $C_p$ of $W(\lambda)$ gives $$v = \sum_{p \in M(\lambda)} \alpha_p C_p$$ for coefficients $\alpha_p \in R$. We may then write $$\chi_{1}(\tau)= \chi_{\sigma}(\tau \sigma) = \langle C_q , v \rangle_{\tau \sigma} =\sum_{p \in M(\lambda)} \alpha_p \langle C_q ,C_p \rangle_{\tau \sigma},$$ which is of the desired form.
\end{proof}

\section{Proof of Theorem \ref{thm:global}} \label{section:GlobalResults}

In this section we will prove Theorem \ref{thm:global}. To do so, we will apply the following observation from \cite{Boyde} inductively.

\begin{theorem}[{\cite[Theorem 3.3]{Boyde}}] \label{thm:quotientingPlus} Let $A$ be an $R$-algebra, let $M$ be a right $A$-module and let $N$ be a left $A$-module. Let $I$ be a twosided ideal of $A$ which acts trivially on $M$ and $N$, and which as a left ideal is a direct sum $I \cong J_1 \oplus \dots \oplus J_k$. Suppose that each $J_i$ is generated as a left ideal by finitely many commuting idempotents. Then
\[\pushQED{\qed} 
\Tor^A_*(M,N) \cong \Tor^{\faktor{A}{I}}_*(M,N). \qedhere
\popQED
\]\end{theorem}

\begin{remark} The hypotheses of this theorem are reminiscent of the definition of a \emph{block} of an associative algebra (see for example \cite[Section 1.8]{Benson}) though of course we are asking for something much weaker than a block decomposition. \end{remark}

\begin{proof}[Proof of Theorem \ref{thm:global}]
We are given
\begin{itemize}
    \item a {\wossit} algebra $A$ over a ring $R$ (commutative with unit), with {\wossit} datum $(\Lambda, G, M, C, *)$, and
    \item two downward closed subsets $Y \subset X$ of $\Lambda$, with $X \setminus Y$ finite.
\end{itemize}

We assume that for each $\lambda \in X \setminus Y$ we have the following.
\begin{itemize}
    \item For every $q \in M(\lambda)$, there exists $v \in W(\lambda)$ such that $$\langle C_q, v \rangle_{\tau} = \begin{cases} 1 & \textrm{ if } \tau=1, \textrm{ and} \\ 0 & \textrm{ otherwise.} \end{cases}$$
\end{itemize}

In this situation, we must show that for any right $A$-module $M$ and left $A$-module $N$, where $I_X$ acts trivially on both, we have
$$\Tor_*^{\faktor{A}{I_{Y}}}(M,N) \cong \Tor_*^{\faktor{A}{I_{X}}}(M,N).$$

Since $X \setminus Y$ is assumed to be finite and $Y$ is arbitrary downward closed, it suffices by induction and Lemma \ref{lem:YcupLambda} to prove that for any minimal element $\lambda$ of $X \setminus Y$ we have $$\Tor_*^{\faktor{A}{I_{Y}}}(M,N) \cong \Tor_*^{\faktor{A}{I_{Y \cup \{\lambda\}}}}(M,N).$$

By Lemma \ref{lem:decomposition}, we have an isomorphism of left $A$-modules $$\faktor{I_{Y \cup \{\lambda \}}}{I_Y} \cong \bigoplus_{q \in M(\lambda)} \pi_X(J_q),$$ so by Theorem \ref{thm:quotientingPlus} applied to the ideal $\faktor{I_{Y \cup \{\lambda \}}}{I_Y}$ of the algebra $\faktor{A}{I_Y}$, it suffices to show that each left ideal $\pi_X(J_q)$ of $\faktor{A}{I_X}$ is generated by a single idempotent, but this follows from the assumption by Proposition \ref{prop:formBegetIdempt} and Lemma \ref{lem:ezCase}. This completes the proof.
\end{proof}

\section{Specialisation of Theorem \ref{thm:global} to subalgebras of the Brauer algebras} \label{section:Specialisation}

Our first job is to give a concrete interpretation of the bilinear form of Definition \ref{def:innerProduct} in the examples. Since all of the examples we consider here are obtained by restriction from the Brauer algebras, it suffices to do so in that case.

\begin{lemma} \label{lem:formInterpretation} Consider the Brauer algebra $\Brauer_n(\delta)$ (Example \ref{ex:Brauer}). Let $t \in \underline{n}$, and let $q,p \in M(t)$ be link states with $t$ defects. Form the pair-graph $\Gamma_{\langle q,p \rangle}$ (Definition \ref{def:pairGraph}).

Recall that the pair-set $S_{\langle q,p \rangle} \subset \pi_0(\Gamma_{\langle q,p \rangle})$ is the intersection of those path components containing a defect vertex of both $q$ and $p$. Note that $S_{\langle q,p \rangle} \leq t$.

The bilinear forms $\langle C_q,C_p \rangle_{\tau}$ in the Brauer algebra $\Brauer_n(\delta)$ are determined as follows:
\begin{itemize}
    \item If $\norm{S_{\langle q,p \rangle}} < t$, then $\langle C_q,C_p \rangle_{\tau} = 0$ for all $\tau$.
    \item If $\norm{S_{\langle q,p \rangle}} = t$, then $\langle C_q,C_p \rangle_{\tau} = \delta^i \chi_{\sigma(p,q)}(\tau)$, for some $\sigma = \sigma(p,q) \in G(t)$, where $i$ is the number of components in $\pi_0(\Gamma_{\langle q,p \rangle})$ not hit by either copy of $\underline{t}$.
\end{itemize}
\end{lemma}

\begin{proof} Combine Lemma \ref{lem:dgrmLikeBilinearForm} and Proposition \ref{prop:BrauerDGL}. \end{proof}

\begin{definition} \label{hypothesisDagger} Let $A$ be a diagram-like subalgebra of $\Brauer_n(\delta)$ on a subset of the diagrams, with naive-cellular datum $(\Lambda, M, G, C, *)$ obtained by restriction (Definition \ref{def:restriction}) from the datum for $\Brauer_n(\delta)$ given in Definition \ref{def:BrauerDGL}. Let $t$ be an integer with $0 \leq t \leq n-1$. We will say that $A$ \emph{satisfies Hypothesis $(\dag)$ at $t$} if:
\begin{itemize}
    \item For all $q \in M(t)$ there exists $p \in M(t)$ such that in the {\sth} $\Gamma_{\langle q,p \rangle}$, we have that the cardinality of the pair-set $\norm{S_{\langle q,p \rangle}}$ is $t$, and that $\delta^i$ is invertible, where $i$ is the number of components in $\pi_0(\Gamma_{\langle q,p \rangle})$ not hit by either copy of $\underline{t}$.
\end{itemize}
\end{definition}

The following theorem specialises Theorem \ref{thm:global} to subalgebras of the Brauer algebras, and is what we will use in our first application. Recall that $I_{\leq t}$ (Definition \ref{def:I}) is the twosided ideal of $A$ spanned by diagrams with at most $t$ left-to-right connections.

\begin{theorem} \label{thm:subalgebrasOfBrauer} Let $A$ be a diagram-like subalgebra of $\Brauer_n(\delta)$ on a subset of the partitions, with naive-cellular datum $(\Lambda, M, G, C, *)$ obtained by restriction from the datum for $\Brauer_n(\delta)$ given in Definition \ref{def:BrauerDGL}. If there exist integers $a$ and $b \leq n-1$ such that $A$ satisfies Hypothesis $(\dag)$ for all $t$ with $a < t \leq b$, then the quotient map induces an isomorphism
$$\Tor_*^{\faktor{A}{I_{\leq a}}}(\t,\t) \cong \Tor_*^{\faktor{A}{I_{\leq b}}}(\t,\t).$$
\end{theorem}

\begin{proof} We wish to apply Theorem \ref{thm:global} with $X = \{t \in \Lambda \mid t \leq b \}$ and $Y = \{t \in \Lambda \mid t \leq a \}$, noting that $b \leq n-1$, so $I_X$ acts trivially on $\t$.

To verify the hypothesis, we must fix $q \in M(t)$ and find $v \in W(t)$ such that $\langle C_q,v\rangle_\tau$, regarded as a function of $\tau$, is the indicator function of some $\sigma \in G(t) = \Sigma_t$.

Letting $p$ be as given in the hypotheses, since $\delta^i$ is invertible, we may form the element $$v = \frac{1}{\delta^i} C_p \in W(t),$$ which will suffice by Lemma \ref{lem:formInterpretation}. \end{proof}

\section{Application: global results for Jones annular algebras} \label{section:Jones}

\subsection{Jones annular algebras}

Recall that the Jones annular algebra $\Jones_n(\delta)$ is diagram-like, with naive-cellular datum obtained by restriction from that of the Brauer algebra $\Brauer_n(\delta)$ (Example \ref{ex:Jones}).

We have the following lemma (see e.g. \cite{Boyde2}):

\begin{lemma} \label{lem:topThingJ} The quotient $\faktor{\Jones_n(\delta)}{I_{\leq n-1}}$ of the topological partition algebra by the twosided ideal $I_{\leq n-1}$ spanned by diagrams having at most $n-1$ left-to-right connections is the group algebra $R C_n$ of the cyclic group of order $n$.  \qed \end{lemma}

Recall the definition of the pair graph $\Gamma_{\langle q, p \rangle}$ (Definition \ref{def:pairGraph}).

\begin{lemma} \label{lem:Jones} Let $q \in M(t)$ be an annular link state in $\Jones_n(\delta)$, for $t \geq 0$. Let $p \in M(t)$ be the link state obtained by `rotating $q$ by one place', so that (mod $n$) if $i$ and $j$ are connected in $q$ then $i+1$ and $j+1$ are connected in $p$.

Then, for each $i \in \underline{n}$, either $i$ has a defect in $q$, or $i$ is connected to $i+1$ in $\Gamma_{\langle q, p \rangle}$. \end{lemma}

\begin{proof} We will show that if $i$ does not have a defect in $q$, then $i$ is connected to $i+1$ in $\Gamma_{\langle q, p \rangle}$.

If $i$ does not have a defect in $q$, then $i$ is connected to some $j$ via an edge in $q$. Since $q$ is annular (Example \ref{ex:Jones}), all of its defects lie in either the cyclic interval $(i,j)$ or in the cyclic interval $(j,i)$. By definition, $p$ has an edge between $i+1$ and $j+1$, so $i$ and $i+1$ are connected in $\Gamma_{\langle q, p \rangle}$ if and only if $j$ and $j+1$ are. We may therefore assume without loss of generality (i.e. perhaps interchanging $i$ and $j$) that $q$ has no defects in the cyclic interval $(i,j)$.

By definition, $p$ has no defects in the cyclic interval $(i+1,j+1)$, so in particular $j$ cannot be a defect in $p$. Since $p$ is annular, $(i+1,j+1)$ is a union of parts of $p$, so $j$ must lie in the same part of $p$ as some $k \in (i+1,j)$ (to which it is connected by an edge in $p$). A picture of this portion of $\Gamma_{\langle q, p\rangle}$, with $q$ on the left and $p$ on the right, is as follows:
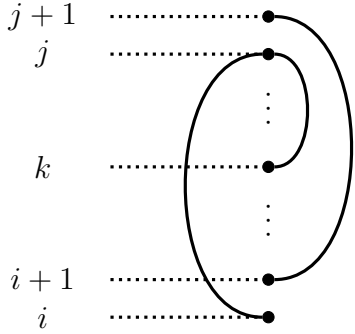
\begin{center}
\quad
\begin{tikzpicture}[x=1.5cm,y=-.5cm,baseline=-1.05cm]

\def\wid{2}

\node[v] (b1) at (1* \wid,0) {};
\node[v] (b2) at (1* \wid,1) {};
\node[v] (b4) at (1* \wid,4) {};
\node[v] (b6) at (1* \wid,7) {};
\node[v] (b7) at (1* \wid,8) {};

\draw (0,0) node{$j+1$};
\draw (0,1) node{$j$};
\draw (0,4) node{$k$};
\draw (0,7) node{$i+1$};
\draw (0,8) node{$i$};

\draw[e, dotted] (0.3* \wid,0) to (b1);
\draw[e, dotted] (0.3* \wid,1) to (b2);
\draw[e, dotted] (0.3* \wid,4) to (b4);
\draw[e, dotted] (0.3* \wid,7) to (b6);
\draw[e, dotted] (0.3* \wid,8) to (b7);

\draw[e] (b1) to[out=0, in=0] (b6);
\draw[e] (b2) to[out=0, in=0] (b4);
\draw[e] (b2) to[out=180, in=180] (b7);

\path (b2) -- (b4) node [midway, sloped] {$\dots$};
\path (b4) -- (b6) node [midway, sloped] {$\dots$};

\end{tikzpicture}
\quad
\end{center}

Continuing, $k$ is connected via an edge in $q$ to some vertex $\ell$, which must be contained in $(i,j)$ and not have a defect, so is connected in turn to some vertex $m$ via an edge in $p$, and so on. This sequence can only terminate by reaching vertex $i+1$ via some edge in $q$.

The conclusion is that $j$ and $i+1$, hence $i$ and $i+1$, lie in the same component of $\Gamma_{\langle q, p \rangle}$, as required. \end{proof}

\begin{corollary} \label{cor:Jones} Let $q \in M(t)$ be an annular link state in $\Jones_n(\delta)$, for $t \geq 0$. Let $p \in M(t)$ be the link state obtained by `rotating $q$ by one place', so that (mod $n$) if $i$ and $j$ are connected in $q$ then $i+1$ and $j+1$ are connected in $p$.

Then, if $t > 0$, $\Gamma_{\langle q, p \rangle}$ consists of $t$ contractible path components, each containing a single defect from $q$ and a single defect from $p$. If $t=0$ then $\Gamma_{\langle q, p \rangle} \simeq S^1$. \end{corollary}

\begin{proof} Let $i \in \underline{n}$ be any vertex. By Lemma \ref{lem:Jones}, either $i$ has a defect in $q$, or $i$ is connected to $i+1$ in $\Gamma_{\langle q, p \rangle}$. By symmetry, we also have that either $i$ has a defect in $p$ or $i$ is connected to $i-1$ in $\Gamma_{\langle q, p \rangle}$. It follows that if $t \geq 1$ then the (preimages in $\underline{n}$ of) connected components of $\Gamma_{\langle q, p \rangle}$ are the cyclic intervals $[i,j]$, where $j$ is a defect in $q$, $i$ is a defect in $p$ and there are no defects in $(i,j)$. Since all vertices have valence 2, this establishes the result for $t \geq 1$.

If $p$ has no defects ($t=0$) then all vertices lie in the same connected component of $\Gamma_{\langle q, p \rangle}$. Again, since all vertices have valence 2, the result follows. \end{proof}

\begin{theorem} \label{thm:JR1} Let $R$ be a commutative ring, and let $\delta \in R$. The Jones annular algebra $\Jones_n(\delta)$ satisfies Hypothesis $(\dag)$ for $1 \leq t \leq n-1$, and if $\delta$ is invertible, then additionally Hypothesis $(\dag)$ holds for $t=0.$ \end{theorem}

\begin{proof} For the first claim, let $q \in M(t)$ for $t \geq 1$. Let $p \in M(t)$ be obtained by rotation, so that if $i$ and $j$ are connected in $q$ then $i+1$ and $j+1$ are connected in $p$.

By Corollary \ref{cor:Jones}, $\norm{S_{\langle q,p\rangle}} = t$, and each path component of $\Gamma_{\langle q, p \rangle}$ is contractible, so $i=0$ by Lemma \ref{lem:formInterpretation}. This establishes Hypothesis $(\dag)$.

When $\delta$ is invertible and $t=0$, we still have $\norm{S_{\langle q,p\rangle}} = t$, so Hypothesis $(\dag)$ again holds. This completes the proof.
\end{proof}

\begin{corollary} \label{cor:JR2} Let $R$ be a commutative ring, let $\delta \in R$, and consider the Jones annular algebra $\Jones_n(\delta)$ (Example \ref{ex:Jones}). Let $I_0$ denote the ideal spanned by partitions $\rho$ such that no part of $\rho$ contains both primed and unprimed elements.

We have
$$\Tor_*^{\faktor{\Jones_n(\delta)}{I_{0}}}(\t,\t) \cong \Tor_*^{R C_n}(\t,\t),$$
and if $\varepsilon$ is invertible or $n$ is odd then additionally
$$\Tor_*^{\Jones_n(\delta)}(\t,\t) \cong \Tor_*^{\faktor{\Jones_n(\delta)}{I_{0}}}(\t,\t).$$
\end{corollary}

\begin{proof} All parts of the claim apart from the case $n$ odd follow from Theorem \ref{thm:subalgebrasOfBrauer}, Lemma \ref{thm:JR1}, and Theorem \ref{lem:topThingJ}.

If $n$ is odd then a partition of $\underline{n} \cup \underline{n}'$ where all parts have cardinality 2 must have at least one part with both a primed and an unprimed element. It follows that $I_0 = 0$, so the result follows. \end{proof}

\section{Link state orderings} \label{section:LSO}

\begin{definition} \label{def:LSOrdering} Let $A$ be a {\wossit} algebra over a ring $R$, with {\wossit} datum $(\Lambda, G, M, C, *)$. A \emph{link state ordering} on the {\wossit} datum is a partial order $<$ on $$M_{\Lambda} := \coprod_{\lambda \in \Lambda} M(\lambda),$$ such that:
\begin{enumerate}
    \item \label{O1} whenever $p,q \in M_{\Lambda}$ have a common lower bound, then they have a greatest common lower bound,
    \item \label{O2} the map $$\pi: \coprod_{\lambda \in \Lambda} M(\lambda) \to \Lambda$$ sending $p \in M(\lambda)$ to $\lambda$ is strictly order-preserving, and
    \item \label{O3} the ordering is compatible with the multiplication in the sense that for all $a \in A$, $\lambda \in \Lambda$, $p,q \in M(\lambda)$ and $\sigma \in G(\lambda)$ we have $$a C_{p,q}^{\sigma} \in \Span_R\{C_{p,q'}^{\sigma} \ \mid \ \mu \leq \lambda, \sigma \in g(\mu), p,q' \in M(\mu), q' \leq q \}$$
\end{enumerate}
By a \emph{link state ordered} {\wossit} algebra, we mean a {\wossit} algebra with a prescribed choice of {\wossit} datum and link state ordering.
\end{definition}

We write $<$ for both the order on $\lambda$ and that on $M_{\lambda}$.

\subsection{Ideals}

In this subsection we fix a {\wossit} algebra over $R$, with {\wossit} datum $(\Lambda, G, M, C, *)$, and a choice of link state ordering.

Recall that for $\lambda \in \Lambda$, and $q \in M(\lambda)$, we define an $R$-submodule $J_q$ of $A$ via 
$$J_q := \Span_R\{C_{p,q}^{\sigma} \ \mid \ \sigma \in g(\lambda), p \in M(\lambda)\}.$$
Using the link state ordering, we now further define $$J_{\leq q} := \Span_R\{C_{p,q'}^{\sigma} \ \mid \ \mu \leq \lambda, \sigma \in g(\mu), p,q' \in M(\mu), q' \leq q \}.$$

Definition \ref{def:LSOrdering}, Condition (\ref{O3}) is then equivalent to:

\begin{lemma} \label{lem} $J_{\leq q}$ is a left ideal of $A$. \qed \end{lemma}

Write $M_{\Lambda} \cup \{-\infty\}$ for the set consisting of $M_{\Lambda}$, together with a single additional element $- \infty$, with the convention that $-\infty < p$ for all $p \in M_{\Lambda}$. By Definition \ref{def:LSOrdering}, Condition (\ref{O1}), we then have:

\begin{lemma} \label{lem:OneMore} $M_{\Lambda} \cup \{-\infty\}$ has all meets. \qed \end{lemma}

By using the convention that $J_{\leq (- \infty)} = J_{- \infty} = 0$, we can avoid a case statement in the next lemma:

\begin{lemma} \label{lem:IntersMeet} For $p,q \in M_{\Lambda}$, we have $J_{\leq p} \cap J_{\leq q} = J_{\leq (p \wedge q)}$ \qed
\end{lemma}

\subsection{Idempotent generation}

In this section we fix a {\wossit} algebra over $R$, with {\wossit} datum $(\Lambda, G, M, C, *)$, and a choice of link state ordering.

\begin{definition} \label{def:generating} A link state ordering will be said to be \emph{(left) generating} if $J_{q}$ generates $J_{\leq q}$ as a left ideal.
\end{definition}

Suppose that $A$ is diagram-like. Then, the product of two basis elements $C_{p_1,q_1}^{\sigma_1} C_{p_2,q_2}^{\sigma_2}$ associated to the same $\lambda \in \Lambda$ is a multiple of some element of the basis of $I_{\leq \lambda}$, associated to some $\lambda'$, and knowledge of $q_1$ and $p_2$ is sufficient to determine whether $\lambda = \lambda'$. Let $S_q \subset M(\lambda)$ be the set of those $p$ for which the product $C_{p_1,q}^{\sigma_1} C_{p,q_2}^{\sigma_2}$ lies in the span of the basis elements associated to $\lambda$ (this question being independent of $p_1, \sigma_1, q_2,$ and $\sigma_2$), and let $$\rho_q : J_q \to J_q$$ be the projection onto those basis elements $C_{p',q}^{\sigma'}$ for which $p' \in S_q$.

\begin{lemma} \label{lem:rhoPreserves} For any $a \in J_q$, right multiplication by $\rho_q(a)$ preserves $J_q$.
\end{lemma}

\begin{proof} It suffices to verify the lemma in the case that $a = C_{p,q}^{\sigma}$ is a basis element. Then $$\rho_q(C_{p,q}^{\sigma})=\begin{cases} C_{p,q}^{\sigma'} & p \in S_q, \textrm{ and} \\ 0 & \textrm{ otherwise.} \end{cases}$$ That is, we must establish that if $p \in S_q$ then right multiplication by $C_{p,q}^{\sigma}$ preserves $J_q$. The definition of a diagram-like algebra gives that: $$C_{p',q}^{\sigma'} C_{p,q}^{\sigma} = \kappa C_{p'',q''}^{\sigma''},$$ and then, since $p \in S_q$, we have either that $\kappa = 0$, or that $C_{p'',q''}^{\sigma''}$ is associated to $\lambda$. In the first case we are done, and in the second case, the definition of a diagram-like algebra (Definition \ref{def:dgrmLike}, Condition \ref{item:dl5}) gives that $q''=q$, as required. \end{proof}

\begin{lemma} \label{lem:IdempotentLift} Suppose that $A$ is diagram-like. Let $\lambda \in \Lambda$, $q \in M(\lambda)$. If the right action of $e \in J_q$ (equivalently, the right action of $\pi(e) \in \pi(J_q)$) fixes the submodule $\pi(J_q) \subset \faktor{A}{I_{< \lambda}}$, then the right action of $\rho_q(e)$ fixes $J_q$. \end{lemma}

\begin{proof} The kernel of $\rho_q$ acts trivially on $\pi(J_q)$ (from the right), so the right multiplication maps $ \cdot \rho_q(e)$ and $ \cdot e$ coincide as maps $\pi(J_p) \to \faktor{A}{I_{< \lambda}}$ on the quotient. By Lemma \ref{lem:rhoPreserves}, the action of $\rho_q(e)$ preserves $J_q$, so we get a commutative diagram
\begin{center}
\begin{tikzcd}
J_q \ar[d,"\pi"] \ar[r,"\cdot \rho_q(e)"] & J_q \ar[d,"\pi"] \\
\pi(J_q) \ar[r,"\cdot e"] & \pi(J_q).
\end{tikzcd}
\end{center}
The vertical map $\pi$ is an isomorphism of $R$-modules, so since the bottom map is assumed to be the identity, the top map is too. This completes the proof.
\end{proof}

\begin{proposition} \label{prop:Idempotents} Suppose that $A$ is diagram-like. Suppose that the link state ordering is left generating (Definition \ref{def:generating}). Let $\lambda \in \Lambda$, $q \in M(\lambda)$. If there exists $v \in W(\lambda)$ such that for some $\sigma \in G(\lambda)$ we have $$\langle C_{q}, v \rangle_{\tau} = \begin{cases} 1 & \textrm{ if } \tau=\sigma, \textrm{ and} \\ 0 & \textrm{ otherwise,} \end{cases}$$
then there exists an idempotent $e_{q}$ in $J_{q}$ which generates $J_{\leq q}$ as a left ideal.
\end{proposition}

\begin{proof} By Proposition \ref{lem:JIdeal}, the image $\pi(J_q)$ in the quotient algebra $\faktor{A}{I_{<\lambda}}$ is a left ideal. By Lemma \ref{lem:ezCase} and Proposition \ref{prop:formBegetIdempt}, the condition of this theorem guarantees that $\pi(J_q)$ is generated by some idempotent $\varepsilon_q$. This means that every element of $\pi(J_q)$ is of the form $x \varepsilon_q$ for $x \in A$, so since $\varepsilon_q$ is idempotent, right multiplication by $\varepsilon_q$ fixes $\pi(J_q)$.

Let $\widetilde{\varepsilon_q} \in J_q$ be the unique element with $\pi(\widetilde{\varepsilon_q}) = \varepsilon_q$, and let $e_q = \rho_q(\widetilde{\varepsilon_q})$.  By Lemma \ref{lem:IdempotentLift}, right multiplication by $e_q$ fixes $J_q$.

By the assumption that the link state ordering is left generating, we have that any $y \in J_{\leq q}$ is of the form $ax$ for $x \in J_q$. Since right multiplication by $e_q$ fixes $J_q$ we have $x = x e_q$, so $y = ax = ax e_q$. That is, $e_q$ generates $J_{\leq q}$ as a left ideal of $A$.

Lastly, $e_q$ is itself in $J_q$, so is fixed by right multiplication by $e_q$: $e_q \cdot e_q = e_q$. This establishes that $e_q$ is idempotent, and completes the proof.
\end{proof}

\section{Application: Homological stability for Temperley-Lieb algebras} \label{section:HS}

In this section, we will convert Sroka's proof of homological stability for the Temperley-Lieb algebras into our language. We wish to use Proposition \ref{prop:Idempotents}, so we must first write down a link state ordering and show that it is left generating.

For the most part, the maneuvers we make with diagrams will be familiar ones: see for example \cite{Kauffman, FGG, FGG-TL, BHComb}.

\begin{definition} \label{def:TLLSO} Let $p$ and $q$ be Temperley-Lieb link states. Say that $p \leq q$ if each connection from $q$ is present in $p$. 

Equivalently, $p \leq q$ if each defect from $p$ is present in $q$, and for each connection in $p$, either this connection is present in $q$, or both of its endpoints are defects in $q$.
\end{definition}

\begin{lemma} \label{lem:TLLSO} The ordering of Definition \ref{def:TLLSO} defines a link state ordering on the Temperley-Lieb algebra $\TL_n(\delta)$. 
\end{lemma}

\begin{proof} We must verify the three conditions of Definition \ref{def:LSOrdering}.

We begin with Condition (\ref{item:C1}). Suppose that link states $p$ and $q$ have a common lower bound. By Definition \ref{def:TLLSO}, this means that each connection in $p$ is either present in $q$ or has endpoints which are defects in $q$, and vice versa. In particular, a connection from $p$ and a connection from $q$ are either equal or disjoint. We will frequently use this fact without comment in the remainder of the proof.

We may therefore let $b$ be the link state having all connections from $p$, all connections from $q$, and defects elsewhere. We must show that $b$ is planar. By the definition of planarity (Definition \ref{def:PlanarLS}) this means we must argue that no two connections from $b$ can cross, and that no defect can lie inside a connection.

Suppose first that two connections cross. Since $p$ and $q$ are planar, it must be that one of these connections is drawn from each. Write $c_p$ for the connection from $p$ and $c_q$ for the connection from $q$. Since $c_p$ is not present in $q$, its endpoints must be defects in $q$. Since $c_p$ and $c_q$ overlap, one of these defects must lie inside $c_q$, but this contradicts planarity of $q$. Thus, no two connections in $b$ can cross.

Suppose then that a defect of $b$ (at height $i$, say) lies inside a connection. Without loss of generality, we may assume that $i$ is a defect in $p$ lying inside a connection $c_q$ in $q$. Since $i$ lies inside $c_q$, by planarity of $q$ we must have that $i$ lies at one end of a connection in $q$. But then, by definition, $i$ must also lie at one end of a connection in $b$, contradicting the assumption that $i$ was a defect in $b$. This establishes that $b$ is planar.

By construction, $b$ is a common lower bound for $p$ and $q$, and we must show that it is the greatest such. If $c$ is another common lower bound, then $c \leq p$ implies that $c$ has all connections from $p$, and $c \leq q$ implies that $c$ has all connections from $q$. This says precisely that $c \leq b$, so $b$ is indeed the greatest common lower bound for $p$ and $q$. This establishes that our ordering satisfies Condition (\ref{item:C1}).

Condition (\ref{item:C2}) follows from the observation that if $p < q$ then $p$ must have strictly fewer defects than $q$.

Condition (\ref{item:C3}) holds, since (the underlying diagram of) a product of Temperley-Lieb diagrams $C_{p',q'}^{\sigma'} C_{p,q}^{\sigma}$ must retain all right-to-right connections from $C_{p,q}^{\sigma}$.

This completes the proof. \end{proof}

\begin{lemma} \label{lem:TLLSOGenerating} The link state ordering of Definition \ref{def:TLLSO} is left generating.
\end{lemma}

\begin{proof} Fix a Temperley-Lieb link state $q$. It suffices to argue that any diagram $C_{p,q'}^{1}$ with right link state $q'$ strictly less than $q$ is a left multiple of a diagram with right link state precisely $q$. Note that by Condition \ref{item:C2} of Definition \ref{def:LSOrdering}, this implies that $q$ has at least one defect (since $q'$ must have strictly fewer).

We will take the liberty of giving the remainder of the proof by pictures. These pictures will be drawn in the case $n=11$ (i.e. in $\TL_{11}(\delta)$), where
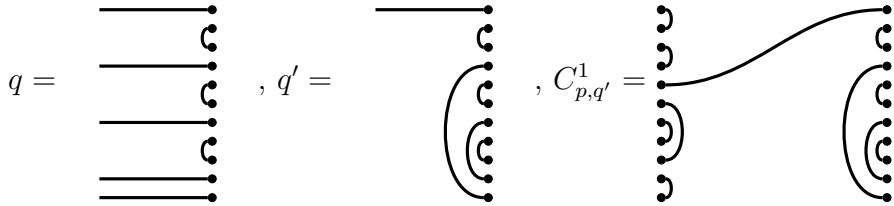
\begin{center}
$q = $
\quad
\begin{tikzpicture}[x=1.5cm,y=-.5cm,baseline=-1.05cm]

\def\wid{2}
\def\hei{0.5}
\def\nodesize{3}
\def\ang{90}

\node[v, minimum size=\nodesize] (c1) at (2* \wid,0*\hei) {};
\node[v, minimum size=\nodesize] (c2) at (2* \wid,1*\hei) {};
\node[v, minimum size=\nodesize] (c3) at (2* \wid,2*\hei) {};
\node[v, minimum size=\nodesize] (c4) at (2* \wid,3*\hei) {};
\node[v, minimum size=\nodesize] (c5) at (2* \wid,4*\hei) {};
\node[v, minimum size=\nodesize] (c6) at (2* \wid,5*\hei) {};
\node[v, minimum size=\nodesize] (c7) at (2* \wid,6*\hei) {};
\node[v, minimum size=\nodesize] (c8) at (2* \wid,7*\hei) {};
\node[v, minimum size=\nodesize] (c9) at (2* \wid,8*\hei) {};
\node[v, minimum size=\nodesize] (c10) at (2* \wid,9*\hei) {};
\node[v, minimum size=\nodesize] (c11) at (2* \wid,10*\hei) {};

\draw[e] (c2) to[out=180, in=180] (c3);
\draw[e] (c6) to[out=180, in=180] (c5);
\draw[e] (c8) to[out=180, in=180] (c9);

\draw[e] (c1) to [out=180,in=0] (1.5*\wid,0*\hei);
\draw[e] (c4) to [out=180,in=0] (1.5*\wid,3*\hei);
\draw[e] (c7) to [out=180,in=0] (1.5*\wid,6*\hei);
\draw[e] (c10) to [out=180,in=0] (1.5*\wid,9*\hei);
\draw[e] (c11) to [out=180,in=0] (1.5*\wid,10*\hei);

\end{tikzpicture}
\quad
, $q' = $
\quad
\begin{tikzpicture}[x=1.5cm,y=-.5cm,baseline=-1.05cm]

\def\wid{2}
\def\hei{0.5}
\def\nodesize{3}
\def\ang{90}

\node[v, minimum size=\nodesize] (c1) at (2* \wid,0*\hei) {};
\node[v, minimum size=\nodesize] (c2) at (2* \wid,1*\hei) {};
\node[v, minimum size=\nodesize] (c3) at (2* \wid,2*\hei) {};
\node[v, minimum size=\nodesize] (c4) at (2* \wid,3*\hei) {};
\node[v, minimum size=\nodesize] (c5) at (2* \wid,4*\hei) {};
\node[v, minimum size=\nodesize] (c6) at (2* \wid,5*\hei) {};
\node[v, minimum size=\nodesize] (c7) at (2* \wid,6*\hei) {};
\node[v, minimum size=\nodesize] (c8) at (2* \wid,7*\hei) {};
\node[v, minimum size=\nodesize] (c9) at (2* \wid,8*\hei) {};
\node[v, minimum size=\nodesize] (c10) at (2* \wid,9*\hei) {};
\node[v, minimum size=\nodesize] (c11) at (2* \wid,10*\hei) {};

\draw[e] (c2) to[out=180, in=180] (c3);
\draw[e] (c6) to[out=180, in=180] (c5);
\draw[e] (c8) to[out=180, in=180] (c9);

\draw[e] (c7) to[out=180, in=180] (c10);
\draw[e] (c4) to[out=180, in=180] (c11);

\draw[e] (c1) to [out=180,in=0] (1.5*\wid,0*\hei);

\end{tikzpicture}
\quad
, $C_{p,q'}^{1}$ = 
\begin{tikzpicture}[x=1.5cm,y=-.5cm,baseline=-1.05cm]

\def\wid{2}
\def\hei{0.5}
\def\nodesize{3}
\def\ang{90}

\node[v, minimum size=\nodesize] (b1) at (1* \wid,0*\hei) {};
\node[v, minimum size=\nodesize] (b2) at (1* \wid,1*\hei) {};
\node[v, minimum size=\nodesize] (b3) at (1* \wid,2*\hei) {};
\node[v, minimum size=\nodesize] (b4) at (1* \wid,3*\hei) {};
\node[v, minimum size=\nodesize] (b5) at (1* \wid,4*\hei) {};
\node[v, minimum size=\nodesize] (b6) at (1* \wid,5*\hei) {};
\node[v, minimum size=\nodesize] (b7) at (1* \wid,6*\hei) {};
\node[v, minimum size=\nodesize] (b8) at (1* \wid,7*\hei) {};
\node[v, minimum size=\nodesize] (b9) at (1* \wid,8*\hei) {};
\node[v, minimum size=\nodesize] (b10) at (1* \wid,9*\hei) {};
\node[v, minimum size=\nodesize] (b11) at (1* \wid,10*\hei) {};

\node[v, minimum size=\nodesize] (c1) at (2* \wid,0*\hei) {};
\node[v, minimum size=\nodesize] (c2) at (2* \wid,1*\hei) {};
\node[v, minimum size=\nodesize] (c3) at (2* \wid,2*\hei) {};
\node[v, minimum size=\nodesize] (c4) at (2* \wid,3*\hei) {};
\node[v, minimum size=\nodesize] (c5) at (2* \wid,4*\hei) {};
\node[v, minimum size=\nodesize] (c6) at (2* \wid,5*\hei) {};
\node[v, minimum size=\nodesize] (c7) at (2* \wid,6*\hei) {};
\node[v, minimum size=\nodesize] (c8) at (2* \wid,7*\hei) {};
\node[v, minimum size=\nodesize] (c9) at (2* \wid,8*\hei) {};
\node[v, minimum size=\nodesize] (c10) at (2* \wid,9*\hei) {};
\node[v, minimum size=\nodesize] (c11) at (2* \wid,10*\hei) {};

\draw[e] (c2) to[out=180, in=180] (c3);
\draw[e] (c6) to[out=180, in=180] (c5);
\draw[e] (c8) to[out=180, in=180] (c9);

\draw[e] (c7) to[out=180, in=180] (c10);
\draw[e] (c4) to[out=180, in=180] (c11);

\draw[e] (c1) to [out=180,in=0] (b5);

\draw[e] (b1) to [out=0,in=0] (b2);
\draw[e] (b3) to [out=0,in=0] (b4);
\draw[e] (b6) to [out=0,in=0] (b9);
\draw[e] (b7) to [out=0,in=0] (b8);
\draw[e] (b10) to [out=0,in=0] (b11);

\end{tikzpicture}
\quad
\end{center}

First `stretch out' the right hand side of $C_{p,q'}^1$, so that the connections which were already present in $q$ stay as they were, and the connections which are new to $q'$ come into the middle. In the example this looks as follows.

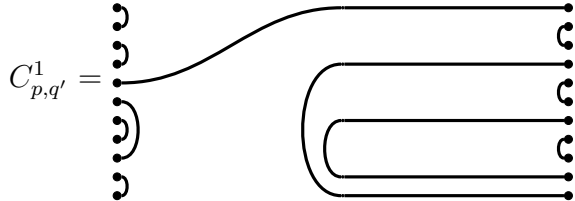
\begin{center}
$C_{p,q'}^{1}$ =
\begin{tikzpicture}[x=1.5cm,y=-.5cm,baseline=-1.05cm]

\def\wid{2}
\def\hei{0.5}
\def\nodesize{3}
\def\ang{90}

\node[v, minimum size=\nodesize] (a1) at (0,0*\hei) {};
\node[v, minimum size=\nodesize] (a2) at (0,1*\hei) {};
\node[v, minimum size=\nodesize] (a3) at (0,2*\hei) {};
\node[v, minimum size=\nodesize] (a4) at (0,3*\hei) {};
\node[v, minimum size=\nodesize] (a5) at (0,4*\hei) {};
\node[v, minimum size=\nodesize] (a6) at (0,5*\hei) {};
\node[v, minimum size=\nodesize] (a7) at (0,6*\hei) {};
\node[v, minimum size=\nodesize] (a8) at (0,7*\hei) {};
\node[v, minimum size=\nodesize] (a9) at (0,8*\hei) {};
\node[v, minimum size=\nodesize] (a10) at (0,9*\hei) {};
\node[v, minimum size=\nodesize] (a11) at (0,10*\hei) {};

\node[u] (b1) at (1* \wid,0*\hei) {};
\node[u] (b4) at (1* \wid,3*\hei) {};
\node[u] (b7) at (1* \wid,6*\hei) {};
\node[u] (b10) at (1* \wid,9*\hei) {};
\node[u] (b11) at (1* \wid,10*\hei) {};

\node[v, minimum size=\nodesize] (c1) at (2* \wid,0*\hei) {};
\node[v, minimum size=\nodesize] (c2) at (2* \wid,1*\hei) {};
\node[v, minimum size=\nodesize] (c3) at (2* \wid,2*\hei) {};
\node[v, minimum size=\nodesize] (c4) at (2* \wid,3*\hei) {};
\node[v, minimum size=\nodesize] (c5) at (2* \wid,4*\hei) {};
\node[v, minimum size=\nodesize] (c6) at (2* \wid,5*\hei) {};
\node[v, minimum size=\nodesize] (c7) at (2* \wid,6*\hei) {};
\node[v, minimum size=\nodesize] (c8) at (2* \wid,7*\hei) {};
\node[v, minimum size=\nodesize] (c9) at (2* \wid,8*\hei) {};
\node[v, minimum size=\nodesize] (c10) at (2* \wid,9*\hei) {};
\node[v, minimum size=\nodesize] (c11) at (2* \wid,10*\hei) {};

\draw[e] (c2) to[out=180, in=180] (c3);
\draw[e] (c6) to[out=180, in=180] (c5);
\draw[e] (c8) to[out=180, in=180] (c9);

\draw[e] (b7) to[out=180, in=180] (b10);
\draw[e] (b4) to[out=180, in=180] (b11);
\draw[e] (c4) to [out=180,in=0] (b4);
\draw[e] (c7) to [out=180,in=0] (b7);
\draw[e] (c10) to [out=180,in=0] (b10);
\draw[e] (c11) to [out=180,in=0] (b11);

\draw[e] (c1) to [out=180,in=0] (b1);
\draw[e] (b1) to [out=180,in=0] (a5);

\draw[e] (a1) to [out=0,in=0] (a2);
\draw[e] (a3) to [out=0,in=0] (a4);
\draw[e] (a6) to [out=0,in=0] (a9);
\draw[e] (a7) to [out=0,in=0] (a8);
\draw[e] (a10) to [out=0,in=0] (a11);

\end{tikzpicture}
\quad
\end{center}

We add vertices at the intersection of these connections with the central vertical, push these new vertices to the bottom, and insert new ones above:
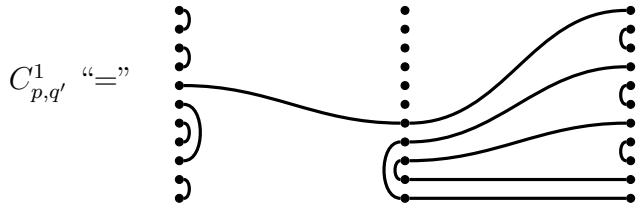
\begin{center}
$C_{p,q'}^{1}$ ``=''
\quad
\begin{tikzpicture}[x=1.5cm,y=-.5cm,baseline=-1.05cm]

\def\wid{2}
\def\hei{0.5}
\def\nodesize{3}
\def\ang{90}

\node[v, minimum size=\nodesize] (a1) at (0,0*\hei) {};
\node[v, minimum size=\nodesize] (a2) at (0,1*\hei) {};
\node[v, minimum size=\nodesize] (a3) at (0,2*\hei) {};
\node[v, minimum size=\nodesize] (a4) at (0,3*\hei) {};
\node[v, minimum size=\nodesize] (a5) at (0,4*\hei) {};
\node[v, minimum size=\nodesize] (a6) at (0,5*\hei) {};
\node[v, minimum size=\nodesize] (a7) at (0,6*\hei) {};
\node[v, minimum size=\nodesize] (a8) at (0,7*\hei) {};
\node[v, minimum size=\nodesize] (a9) at (0,8*\hei) {};
\node[v, minimum size=\nodesize] (a10) at (0,9*\hei) {};
\node[v, minimum size=\nodesize] (a11) at (0,10*\hei) {};

\node[v, minimum size=\nodesize] (b1) at (1* \wid,0*\hei) {};
\node[v, minimum size=\nodesize] (b2) at (1* \wid,1*\hei) {};
\node[v, minimum size=\nodesize] (b3) at (1* \wid,2*\hei) {};
\node[v, minimum size=\nodesize] (b4) at (1* \wid,3*\hei) {};
\node[v, minimum size=\nodesize] (b5) at (1* \wid,4*\hei) {};
\node[v, minimum size=\nodesize] (b6) at (1* \wid,5*\hei) {};
\node[v, minimum size=\nodesize] (b7) at (1* \wid,6*\hei) {};
\node[v, minimum size=\nodesize] (b8) at (1* \wid,7*\hei) {};
\node[v, minimum size=\nodesize] (b9) at (1* \wid,8*\hei) {};
\node[v, minimum size=\nodesize] (b10) at (1* \wid,9*\hei) {};
\node[v, minimum size=\nodesize] (b11) at (1* \wid,10*\hei) {};

\node[v, minimum size=\nodesize] (c1) at (2* \wid,0*\hei) {};
\node[v, minimum size=\nodesize] (c2) at (2* \wid,1*\hei) {};
\node[v, minimum size=\nodesize] (c3) at (2* \wid,2*\hei) {};
\node[v, minimum size=\nodesize] (c4) at (2* \wid,3*\hei) {};
\node[v, minimum size=\nodesize] (c5) at (2* \wid,4*\hei) {};
\node[v, minimum size=\nodesize] (c6) at (2* \wid,5*\hei) {};
\node[v, minimum size=\nodesize] (c7) at (2* \wid,6*\hei) {};
\node[v, minimum size=\nodesize] (c8) at (2* \wid,7*\hei) {};
\node[v, minimum size=\nodesize] (c9) at (2* \wid,8*\hei) {};
\node[v, minimum size=\nodesize] (c10) at (2* \wid,9*\hei) {};
\node[v, minimum size=\nodesize] (c11) at (2* \wid,10*\hei) {};

\draw[e] (c2) to[out=180, in=180] (c3);
\draw[e] (c6) to[out=180, in=180] (c5);
\draw[e] (c8) to[out=180, in=180] (c9);

\draw[e] (b9) to[out=180, in=180] (b10);
\draw[e] (b8) to[out=180, in=180] (b11);
\draw[e] (c4) to [out=180,in=0] (b8);
\draw[e] (c7) to [out=180,in=0] (b9);
\draw[e] (c10) to [out=180,in=0] (b10);
\draw[e] (c11) to [out=180,in=0] (b11);

\draw[e] (c1) to [out=180,in=0] (b7);
\draw[e] (b7) to [out=180,in=0] (a5);

\draw[e] (a1) to [out=0,in=0] (a2);
\draw[e] (a3) to [out=0,in=0] (a4);
\draw[e] (a6) to [out=0,in=0] (a9);
\draw[e] (a7) to [out=0,in=0] (a8);
\draw[e] (a10) to [out=0,in=0] (a11);

\end{tikzpicture}
\quad
\end{center}

The number of new vertices inserted is $n-t$, where $t$ is the number of defects in $q$ (i.e. $q \in M(t)$). It is automatic that $n-t$ must be even. As noted at the start of the proof, $t \geq 1$, so there exists at least one left-to-right connection in the right-hand half of this `diagram'. The topmost such connection may now take a `detour' through the remaining vertices, since there are an even number of them:

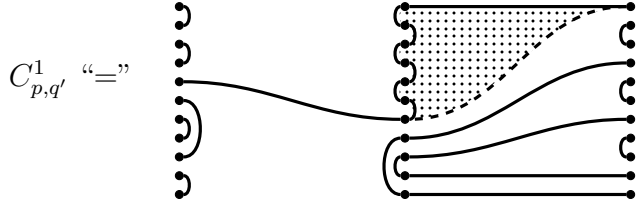
\begin{center}
$C_{p,q'}^{1}$ ``=''
\quad
\begin{tikzpicture}[x=1.5cm,y=-.5cm,baseline=-1.05cm]

\def\wid{2}
\def\hei{0.5}
\def\nodesize{3}
\def\ang{90}

\node[v, minimum size=\nodesize] (a1) at (0,0*\hei) {};
\node[v, minimum size=\nodesize] (a2) at (0,1*\hei) {};
\node[v, minimum size=\nodesize] (a3) at (0,2*\hei) {};
\node[v, minimum size=\nodesize] (a4) at (0,3*\hei) {};
\node[v, minimum size=\nodesize] (a5) at (0,4*\hei) {};
\node[v, minimum size=\nodesize] (a6) at (0,5*\hei) {};
\node[v, minimum size=\nodesize] (a7) at (0,6*\hei) {};
\node[v, minimum size=\nodesize] (a8) at (0,7*\hei) {};
\node[v, minimum size=\nodesize] (a9) at (0,8*\hei) {};
\node[v, minimum size=\nodesize] (a10) at (0,9*\hei) {};
\node[v, minimum size=\nodesize] (a11) at (0,10*\hei) {};

\node[v, minimum size=\nodesize] (b1) at (1* \wid,0*\hei) {};
\node[v, minimum size=\nodesize] (b2) at (1* \wid,1*\hei) {};
\node[v, minimum size=\nodesize] (b3) at (1* \wid,2*\hei) {};
\node[v, minimum size=\nodesize] (b4) at (1* \wid,3*\hei) {};
\node[v, minimum size=\nodesize] (b5) at (1* \wid,4*\hei) {};
\node[v, minimum size=\nodesize] (b6) at (1* \wid,5*\hei) {};
\node[v, minimum size=\nodesize] (b7) at (1* \wid,6*\hei) {};
\node[v, minimum size=\nodesize] (b8) at (1* \wid,7*\hei) {};
\node[v, minimum size=\nodesize] (b9) at (1* \wid,8*\hei) {};
\node[v, minimum size=\nodesize] (b10) at (1* \wid,9*\hei) {};
\node[v, minimum size=\nodesize] (b11) at (1* \wid,10*\hei) {};

\node[v, minimum size=\nodesize] (c1) at (2* \wid,0*\hei) {};
\node[v, minimum size=\nodesize] (c2) at (2* \wid,1*\hei) {};
\node[v, minimum size=\nodesize] (c3) at (2* \wid,2*\hei) {};
\node[v, minimum size=\nodesize] (c4) at (2* \wid,3*\hei) {};
\node[v, minimum size=\nodesize] (c5) at (2* \wid,4*\hei) {};
\node[v, minimum size=\nodesize] (c6) at (2* \wid,5*\hei) {};
\node[v, minimum size=\nodesize] (c7) at (2* \wid,6*\hei) {};
\node[v, minimum size=\nodesize] (c8) at (2* \wid,7*\hei) {};
\node[v, minimum size=\nodesize] (c9) at (2* \wid,8*\hei) {};
\node[v, minimum size=\nodesize] (c10) at (2* \wid,9*\hei) {};
\node[v, minimum size=\nodesize] (c11) at (2* \wid,10*\hei) {};

\draw[e] (c2) to[out=180, in=180] (c3);
\draw[e] (c6) to[out=180, in=180] (c5);
\draw[e] (c8) to[out=180, in=180] (c9);

\draw[e] (b9) to[out=180, in=180] (b10);
\draw[e] (b8) to[out=180, in=180] (b11);
\draw[e] (c4) to [out=180,in=0] (b8);
\draw[e] (c7) to [out=180,in=0] (b9);
\draw[e] (c10) to [out=180,in=0] (b10);
\draw[e] (c11) to [out=180,in=0] (b11);

\draw[e] (c1) to [out=180,in=0] (b1);
\draw[e] (b1) to[out=180, in=180] (b2);
\draw[e] (b2) to[out=0, in=0] (b3);
\draw[e] (b3) to[out=180, in=180] (b4);
\draw[e] (b4) to[out=0, in=0] (b5);
\draw[e] (b5) to[out=180, in=180] (b6);
\draw[e] (b6) to[out=0, in=0] (b7);
\draw[e] (b7) to [out=180,in=0] (a5);

\draw[pattern=dots,draw=none] (2* \wid,0*\hei) -- (1* \wid,0*\hei) to[out=180, in=180] (1* \wid,1*\hei) to[out=0, in=0] (1* \wid,2*\hei) to[out=180, in=180] (1* \wid,3*\hei) to[out=0, in=0] (1* \wid,4*\hei) to[out=180, in=180] (1* \wid,5*\hei) to[out=0, in=0] (1* \wid,6*\hei) to [out=0,in=180] (2* \wid,0*\hei);

\draw[e,dashed] (c1) to [out=180,in=0] (b7);

\draw[e] (a1) to [out=0,in=0] (a2);
\draw[e] (a3) to [out=0,in=0] (a4);
\draw[e] (a6) to [out=0,in=0] (a9);
\draw[e] (a7) to [out=0,in=0] (a8);
\draw[e] (a10) to [out=0,in=0] (a11);

\end{tikzpicture}
\quad
\end{center}

This expresses $C_{p,q}^1$ as a product of the desired form, namely:

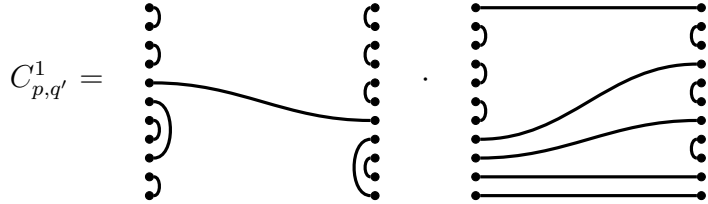
\begin{center}
$C_{p,q'}^{1}$ =
\quad
\begin{tikzpicture}[x=1.5cm,y=-.5cm,baseline=-1.05cm]

\def\wid{2}
\def\hei{0.5}
\def\nodesize{3}
\def\ang{90}

\node[v, minimum size=\nodesize] (a1) at (0,0*\hei) {};
\node[v, minimum size=\nodesize] (a2) at (0,1*\hei) {};
\node[v, minimum size=\nodesize] (a3) at (0,2*\hei) {};
\node[v, minimum size=\nodesize] (a4) at (0,3*\hei) {};
\node[v, minimum size=\nodesize] (a5) at (0,4*\hei) {};
\node[v, minimum size=\nodesize] (a6) at (0,5*\hei) {};
\node[v, minimum size=\nodesize] (a7) at (0,6*\hei) {};
\node[v, minimum size=\nodesize] (a8) at (0,7*\hei) {};
\node[v, minimum size=\nodesize] (a9) at (0,8*\hei) {};
\node[v, minimum size=\nodesize] (a10) at (0,9*\hei) {};
\node[v, minimum size=\nodesize] (a11) at (0,10*\hei) {};

\node[v, minimum size=\nodesize] (b1) at (1* \wid,0*\hei) {};
\node[v, minimum size=\nodesize] (b2) at (1* \wid,1*\hei) {};
\node[v, minimum size=\nodesize] (b3) at (1* \wid,2*\hei) {};
\node[v, minimum size=\nodesize] (b4) at (1* \wid,3*\hei) {};
\node[v, minimum size=\nodesize] (b5) at (1* \wid,4*\hei) {};
\node[v, minimum size=\nodesize] (b6) at (1* \wid,5*\hei) {};
\node[v, minimum size=\nodesize] (b7) at (1* \wid,6*\hei) {};
\node[v, minimum size=\nodesize] (b8) at (1* \wid,7*\hei) {};
\node[v, minimum size=\nodesize] (b9) at (1* \wid,8*\hei) {};
\node[v, minimum size=\nodesize] (b10) at (1* \wid,9*\hei) {};
\node[v, minimum size=\nodesize] (b11) at (1* \wid,10*\hei) {};

\draw[e] (b9) to[out=180, in=180] (b10);
\draw[e] (b8) to[out=180, in=180] (b11);

\draw[e] (b1) to[out=180, in=180] (b2);
\draw[e] (b3) to[out=180, in=180] (b4);
\draw[e] (b5) to[out=180, in=180] (b6);
\draw[e] (b7) to [out=180,in=0] (a5);

\draw[e] (a1) to [out=0,in=0] (a2);
\draw[e] (a3) to [out=0,in=0] (a4);
\draw[e] (a6) to [out=0,in=0] (a9);
\draw[e] (a7) to [out=0,in=0] (a8);
\draw[e] (a10) to [out=0,in=0] (a11);

\end{tikzpicture}
\quad
$\cdot$
\quad
\begin{tikzpicture}[x=1.5cm,y=-.5cm,baseline=-1.05cm]

\def\wid{2}
\def\hei{0.5}
\def\nodesize{3}
\def\ang{90}

\node[v, minimum size=\nodesize] (b1) at (1* \wid,0*\hei) {};
\node[v, minimum size=\nodesize] (b2) at (1* \wid,1*\hei) {};
\node[v, minimum size=\nodesize] (b3) at (1* \wid,2*\hei) {};
\node[v, minimum size=\nodesize] (b4) at (1* \wid,3*\hei) {};
\node[v, minimum size=\nodesize] (b5) at (1* \wid,4*\hei) {};
\node[v, minimum size=\nodesize] (b6) at (1* \wid,5*\hei) {};
\node[v, minimum size=\nodesize] (b7) at (1* \wid,6*\hei) {};
\node[v, minimum size=\nodesize] (b8) at (1* \wid,7*\hei) {};
\node[v, minimum size=\nodesize] (b9) at (1* \wid,8*\hei) {};
\node[v, minimum size=\nodesize] (b10) at (1* \wid,9*\hei) {};
\node[v, minimum size=\nodesize] (b11) at (1* \wid,10*\hei) {};

\node[v, minimum size=\nodesize] (c1) at (2* \wid,0*\hei) {};
\node[v, minimum size=\nodesize] (c2) at (2* \wid,1*\hei) {};
\node[v, minimum size=\nodesize] (c3) at (2* \wid,2*\hei) {};
\node[v, minimum size=\nodesize] (c4) at (2* \wid,3*\hei) {};
\node[v, minimum size=\nodesize] (c5) at (2* \wid,4*\hei) {};
\node[v, minimum size=\nodesize] (c6) at (2* \wid,5*\hei) {};
\node[v, minimum size=\nodesize] (c7) at (2* \wid,6*\hei) {};
\node[v, minimum size=\nodesize] (c8) at (2* \wid,7*\hei) {};
\node[v, minimum size=\nodesize] (c9) at (2* \wid,8*\hei) {};
\node[v, minimum size=\nodesize] (c10) at (2* \wid,9*\hei) {};
\node[v, minimum size=\nodesize] (c11) at (2* \wid,10*\hei) {};

\draw[e] (c2) to[out=180, in=180] (c3);
\draw[e] (c6) to[out=180, in=180] (c5);
\draw[e] (c8) to[out=180, in=180] (c9);

\draw[e] (c4) to [out=180,in=0] (b8);
\draw[e] (c7) to [out=180,in=0] (b9);
\draw[e] (c10) to [out=180,in=0] (b10);
\draw[e] (c11) to [out=180,in=0] (b11);

\draw[e] (c1) to [out=180,in=0] (b1);
\draw[e] (b2) to[out=0, in=0] (b3);
\draw[e] (b4) to[out=0, in=0] (b5);
\draw[e] (b6) to[out=0, in=0] (b7);

\end{tikzpicture}
\quad
\end{center}

This completes the proof. \end{proof}

\subsection{The idempotent cover}

In \cite{Sroka}, Sroka constructs a resolution of the trivial module $\t$ by projective left $\TL_n(\delta)$-modules. In the language of \cite{Boyde2}, his complex can be described as an \emph{idempotent cover} (indeed, it was the motivation for introducing that terminology). In this section, we will reconstruct Sroka's complex in this language.

For $1 \leq i \leq n-1$, let $K_i$ be the left ideal of $\TL_n(\delta)$ spanned by those diagrams which have a connection between the (right-hand) nodes $i'$ and $(i+1)'$, and let $I_{\leq n-1}$ be the twosided ideal spanned by diagrams with fewer than $n$ left-to-right connections.

\begin{lemma} There is an equality of left ideals $K_1 + \dots + K_{n-1} = I_{\leq n-1}$, where the sum is the internal sum in $\TL_n(\delta)$.
\end{lemma}

In the language of \cite{Boyde2}, this says that the ideals $K_i$ \emph{cover} $I_{\leq n-1}$.

\begin{proof} Any diagram with a right-to-right connection must have fewer than $n$ primed vertices with a left-to-right connection, hence in particular fewer than $n$ left-to-right connections. This gives that $K_i \subset I_{\leq n-1}$ for each $i$, and it suffices to show that any diagram with fewer than $n$ left-to-right connections must be contained in some $K_i$.

Any diagram having fewer than $n$ left-to-right connections must have a primed vertex without a left-to-right connection. This vertex, $j_0'$, say, must be connected via a right-to-right connection to some other primed vertex $k_0'$. If $|j_0-k_0|=1$ then we are done, otherwise, since the right link state is planar (Definition \ref{def:PlanarLS}), the primed vertices lying between $j_0'$ and $k_0'$ can only be connected to one another. Choose a connection among this set, say from $j_1$ to $k_1$, and repeat. Necessarily at each stage we have $|j_\ell-k_\ell| < |j_{\ell-1}-k_{\ell-1}|$, so we must eventually reach a connection between adjacent vertices $i$ and $i+1$, which shows that the diagram lies in $K_i$, as required.
\end{proof}

Following Sroka, call a subset $S$ of $\underline{n}$ \emph{innermost} if there exists no $i$ for which $i$ and $i+1$ are both elements of $S$. The following lemma is immediate.

\begin{lemma} \label{lem:intersectionDescription} For $S \subset \underline{n}$, the intersection $$\bigcap_{i \in S} K_i$$ is the $R$-span of those diagrams having a connection from $i'$ to $(i+1)'$ for each $i \in S$. This intersection is non-zero if and only if $S$ is innermost. \qed
\end{lemma}

\begin{lemma} \label{lem:intersectionsToJ} For $S \subset \underline{n}$ innermost, let $q=q(S)$ be the link state having a connection from $i$ to $i+1$ whenever $i \in S$, and defects elsewhere. We have an equality of left ideals $$\bigcap_{i \in S} K_i = J_{\leq q}.$$ \end{lemma}

\begin{proof} To say that a diagram $C_{p',q'}^1$ lies in $\bigcap_{i \in S} K_i$ is precisely to say that $q' \leq q$ in the link state ordering (Definition \ref{def:TLLSO}). Since the link state ordering is left generating (Lemma \ref{lem:TLLSOGenerating}), and since link state orderings respect the multiplication (Definition \ref{def:LSOrdering}, Condition (\ref{item:C3})) this is equivalent to saying that $C_{p',q'}^1$ lies in $J_{\leq q(S)}$. \end{proof}

\begin{proposition} \label{prop:wiggleWiggle} Let $t \geq 1$, $q \in M(t)$. There exists $p \in M(t)$ such that the pair-graph $\Gamma_{\langle q, p \rangle}$ has no loops and such that the cardinality of the pair-set $\norm{S_{\langle q,p\rangle}}$ is $t$.
\end{proposition}

After the proof, we will give an example of the algorithm.

\begin{proof} We will build $p$ inductively, adding one edge or defect at a time.

Let $L \subset \underline{n}$ be the set of \emph{live vertices}. Initially, set $L$ equal to the defects of $q$. Since we assume $t \geq 1$, $L$ is initially non-empty. We will also describe each edge of $q$ as either \emph{available} or \emph{unavailable}. Write $A$ for the set of available edges. Initially, all edges are available.

At each stage of the algorithm, select a live vertex $i \in L \subset \underline{n}$. If there exists at least one $j \in \underline{n}$ such that
\begin{itemize}
    \item $j$ lies at one end of an available edge $e$ in $q$, and
    \item no live vertex $i' \in L$ lies between $i$ and $j$
\end{itemize}
then select from among such $j$ one closest to $i$, and
\begin{itemize}
    \item add a connection in $p$ from $i$ to $j$,
    \item replace $i$ in the set of live vertices by the vertex $k$ lying at the other end of $e$ from $j$, and
    \item remove $e$ from the set of available edges.
\end{itemize}
If no such $j$ exists, then add a defect at $i$, and remove $i$ from the set $L$ of live vertices.

This algorithm terminates when the set of live vertices is empty.

We must argue that after performing this algorithm all vertices have exactly one edge in $p$ (so that $p$ is a Brauer link state), that $p$ is planar (hence a Temperley-Lieb link state) and that the conditions of the proposition statement are satisfied.

If at some stage at least one available edge remains, then it is automatic that it satisfies the criterion of the algorithm for some live vertex (i.e. one of possibly two closest to it). This means that the algorithm terminates precisely when all edges have become unavailable and all remaining live vertices have been converted to defects.

For a given edge $e$ from $i$ to $j$ in $q$, we see that $e$ starts the algorithm available, and can receive a connection only while it is available. Upon receiving a connection (at $i$, say), the other end $j$ of $e$ then becomes a new live vertex, and must later either become a defect or be connected to something else. The algorithm will produce no further connections to $i$ or $j$, since $e$ is now unavailable. Since the algorithm terminates when all edges have become unavailable and all remaining live vertices have been converted to defects, it follows that in $p$, the ends $i$ and $j$ of $e$ either both have connections in $p$, or one has a connection and one has a defect. In particular, this says precisely that the resulting $p$ is a Brauer link state.

By induction, the unavailable edges at each stage of the algorithm are precisely those which are connected by a sequence of edges in $\Gamma_{\langle q,p \rangle}$ to some defect of $q$. Also by induction, an unavailable edge must be connected at each stage by a sequence of edges to some live vertex, hence ultimately to some defect of $p$. Since all edges are ultimately unavailable, all edges of $q$, hence all vertices in $\underline{n}$, are ultimately connected to both a defect of $q$ and a defect of $p$. This gives that the cardinality of the pair set $\norm{S_{\langle q,p\rangle}}$ is $t$, and that the pair graph $\Gamma_{\langle q,p \rangle}$ has no loops, as required.

It remains to argue that $p$ is planar (Definition \ref{def:PlanarLS}). We must argue for each edge $e$ from $i$ to $j$ (say $i<j$) in $p$:
\begin{itemize}
    \item no defects of $p$ lie in the interval $(i,j)$, and
    \item no vertex in $(i,j)$ is connected to a vertex outside $(i,j)$.
\end{itemize}

At each stage, the algorithm begins from a live vertex and selects a nearest vertex lying at the end of an available edge in $q$, with no live vertices in between. In particular, after adding this new edge $e$ to $p$, no available edges of $q$ have an end lying inside $e$. It follows that $(i,j)$ consists only of vertices which already had an edge in $p$ when $e$ was added. Such an edge $e'$ must necessarily go to another vertex in the interval $(i,j)$, because otherwise when $e'$ was added there was a closer available vertex, contradicting the definition of the algorithm.

In total, this shows that the interval $(i,j)$ consists entirely of vertices connected by an edge to another vertex in $(i,j)$. In particular, no defects can be created in this interval because it contains no live vertices, and no live vertices can be created in it, all of its edges being unavailable. This establishes that $p$ is planar, and completes the proof. \end{proof}

\begin{example} We will give an example of the algorithm, with $n=16$, and
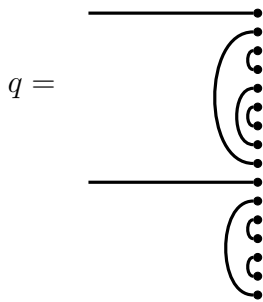
\begin{center}
$q=$\quad
\begin{tikzpicture}[x=1.5cm,y=-.5cm,baseline=-1.05cm]

\def\wid{2}
\def\hei{0.5}
\def\nodesize{3}
\def\ang{90}

\node[v, minimum size=\nodesize] (c1) at (2* \wid,0*\hei) {};
\node[v, minimum size=\nodesize] (c2) at (2* \wid,1*\hei) {};
\node[v, minimum size=\nodesize] (c3) at (2* \wid,2*\hei) {};
\node[v, minimum size=\nodesize] (c4) at (2* \wid,3*\hei) {};
\node[v, minimum size=\nodesize] (c5) at (2* \wid,4*\hei) {};
\node[v, minimum size=\nodesize] (c6) at (2* \wid,5*\hei) {};
\node[v, minimum size=\nodesize] (c7) at (2* \wid,6*\hei) {};
\node[v, minimum size=\nodesize] (c8) at (2* \wid,7*\hei) {};
\node[v, minimum size=\nodesize] (c9) at (2* \wid,8*\hei) {};
\node[v, minimum size=\nodesize] (c10) at (2* \wid,9*\hei) {};
\node[v, minimum size=\nodesize] (c11) at (2* \wid,10*\hei) {};
\node[v, minimum size=\nodesize] (c12) at (2* \wid,11*\hei) {};
\node[v, minimum size=\nodesize] (c13) at (2* \wid,12*\hei) {};
\node[v, minimum size=\nodesize] (c14) at (2* \wid,13*\hei) {};
\node[v, minimum size=\nodesize] (c15) at (2* \wid,14*\hei) {};
\node[v, minimum size=\nodesize] (c16) at (2* \wid,15*\hei) {};

\draw[e] (c1) to [out=180,in=0] (1.25*\wid,0*\hei);

\draw[e] (c2) to[out=180, in=180] (c9);
\draw[e] (c3) to[out=180, in=180] (c4);
\draw[e] (c5) to[out=180, in=180] (c8);
\draw[e] (c6) to[out=180, in=180] (c7);

\draw[e] (c10) to [out=180,in=0] (1.25*\wid,9*\hei);

\draw[e] (c11) to[out=180, in=180] (c16);
\draw[e] (c12) to[out=180, in=180] (c13);
\draw[e] (c14) to[out=180, in=180] (c15);

\end{tikzpicture}
\quad
\end{center}

Live vertices will be drawn in red, and unavailable edges (and the edges of $p$) will be drawn in grey. Note that the algorithm involves a choice at each stage. One sequence of choices leads to the following construction of $p$:

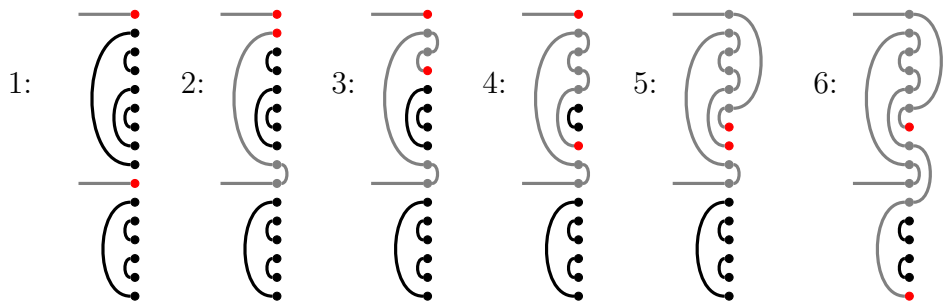
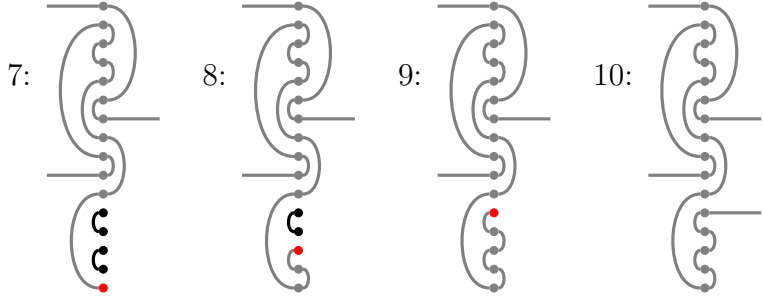
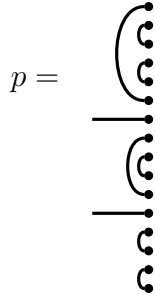
\begin{center}
1:
\quad
\begin{tikzpicture}[x=1.5cm,y=-.5cm,baseline=-1.05cm]

\def\wid{2}
\def\hei{0.5}
\def\nodesize{3}
\def\ang{90}

\node[v, red, minimum size=\nodesize] (c1) at (2* \wid,0*\hei) {};
\node[v, minimum size=\nodesize] (c2) at (2* \wid,1*\hei) {};
\node[v, minimum size=\nodesize] (c3) at (2* \wid,2*\hei) {};
\node[v, minimum size=\nodesize] (c4) at (2* \wid,3*\hei) {};
\node[v, minimum size=\nodesize] (c5) at (2* \wid,4*\hei) {};
\node[v, minimum size=\nodesize] (c6) at (2* \wid,5*\hei) {};
\node[v, minimum size=\nodesize] (c7) at (2* \wid,6*\hei) {};
\node[v, minimum size=\nodesize] (c8) at (2* \wid,7*\hei) {};
\node[v, minimum size=\nodesize] (c9) at (2* \wid,8*\hei) {};
\node[v, red, minimum size=\nodesize] (c10) at (2* \wid,9*\hei) {};
\node[v, minimum size=\nodesize] (c11) at (2* \wid,10*\hei) {};
\node[v, minimum size=\nodesize] (c12) at (2* \wid,11*\hei) {};
\node[v, minimum size=\nodesize] (c13) at (2* \wid,12*\hei) {};
\node[v, minimum size=\nodesize] (c14) at (2* \wid,13*\hei) {};
\node[v, minimum size=\nodesize] (c15) at (2* \wid,14*\hei) {};
\node[v, minimum size=\nodesize] (c16) at (2* \wid,15*\hei) {};

\draw[e,gray] (c1) to [out=180,in=0] (1.75*\wid,0*\hei);

\draw[e] (c2) to[out=180, in=180] (c9);
\draw[e] (c3) to[out=180, in=180] (c4);
\draw[e] (c5) to[out=180, in=180] (c8);
\draw[e] (c6) to[out=180, in=180] (c7);

\draw[e,gray] (c10) to [out=180,in=0] (1.75*\wid,9*\hei);

\draw[e] (c11) to[out=180, in=180] (c16);
\draw[e] (c12) to[out=180, in=180] (c13);
\draw[e] (c14) to[out=180, in=180] (c15);

\end{tikzpicture}
\quad
2:
\begin{tikzpicture}[x=1.5cm,y=-.5cm,baseline=-1.05cm]

\def\wid{2}
\def\hei{0.5}
\def\nodesize{3}
\def\ang{90}

\node[v, red, minimum size=\nodesize] (c1) at (2* \wid,0*\hei) {};
\node[v, red, minimum size=\nodesize] (c2) at (2* \wid,1*\hei) {};
\node[v, minimum size=\nodesize] (c3) at (2* \wid,2*\hei) {};
\node[v, minimum size=\nodesize] (c4) at (2* \wid,3*\hei) {};
\node[v, minimum size=\nodesize] (c5) at (2* \wid,4*\hei) {};
\node[v, minimum size=\nodesize] (c6) at (2* \wid,5*\hei) {};
\node[v, minimum size=\nodesize] (c7) at (2* \wid,6*\hei) {};
\node[v, minimum size=\nodesize] (c8) at (2* \wid,7*\hei) {};
\node[v, gray, minimum size=\nodesize] (c9) at (2* \wid,8*\hei) {};
\node[v, gray, minimum size=\nodesize] (c10) at (2* \wid,9*\hei) {};
\node[v, minimum size=\nodesize] (c11) at (2* \wid,10*\hei) {};
\node[v, minimum size=\nodesize] (c12) at (2* \wid,11*\hei) {};
\node[v, minimum size=\nodesize] (c13) at (2* \wid,12*\hei) {};
\node[v, minimum size=\nodesize] (c14) at (2* \wid,13*\hei) {};
\node[v, minimum size=\nodesize] (c15) at (2* \wid,14*\hei) {};
\node[v, minimum size=\nodesize] (c16) at (2* \wid,15*\hei) {};

\draw[e,gray] (c1) to [out=180,in=0] (1.75*\wid,0*\hei);

\draw[e,gray] (c2) to[out=180, in=180] (c9);
\draw[e] (c3) to[out=180, in=180] (c4);
\draw[e] (c5) to[out=180, in=180] (c8);
\draw[e] (c6) to[out=180, in=180] (c7);

\draw[e,gray] (c10) to [out=180,in=0] (1.75*\wid,9*\hei);

\draw[e] (c11) to[out=180, in=180] (c16);
\draw[e] (c12) to[out=180, in=180] (c13);
\draw[e] (c14) to[out=180, in=180] (c15);

\draw[e,gray] (c10) to[out=0, in=0] (c9);

\end{tikzpicture}
\quad
3:
\begin{tikzpicture}[x=1.5cm,y=-.5cm,baseline=-1.05cm]

\def\wid{2}
\def\hei{0.5}
\def\nodesize{3}
\def\ang{90}

\node[v, red, minimum size=\nodesize] (c1) at (2* \wid,0*\hei) {};
\node[v, gray, minimum size=\nodesize] (c2) at (2* \wid,1*\hei) {};
\node[v, gray, minimum size=\nodesize] (c3) at (2* \wid,2*\hei) {};
\node[v, red, minimum size=\nodesize] (c4) at (2* \wid,3*\hei) {};
\node[v, minimum size=\nodesize] (c5) at (2* \wid,4*\hei) {};
\node[v, minimum size=\nodesize] (c6) at (2* \wid,5*\hei) {};
\node[v, minimum size=\nodesize] (c7) at (2* \wid,6*\hei) {};
\node[v, minimum size=\nodesize] (c8) at (2* \wid,7*\hei) {};
\node[v, gray, minimum size=\nodesize] (c9) at (2* \wid,8*\hei) {};
\node[v, gray, minimum size=\nodesize] (c10) at (2* \wid,9*\hei) {};
\node[v, minimum size=\nodesize] (c11) at (2* \wid,10*\hei) {};
\node[v, minimum size=\nodesize] (c12) at (2* \wid,11*\hei) {};
\node[v, minimum size=\nodesize] (c13) at (2* \wid,12*\hei) {};
\node[v, minimum size=\nodesize] (c14) at (2* \wid,13*\hei) {};
\node[v, minimum size=\nodesize] (c15) at (2* \wid,14*\hei) {};
\node[v, minimum size=\nodesize] (c16) at (2* \wid,15*\hei) {};

\draw[e,gray] (c1) to [out=180,in=0] (1.75*\wid,0*\hei);

\draw[e,gray] (c2) to[out=180, in=180] (c9);
\draw[e,gray] (c3) to[out=180, in=180] (c4);
\draw[e] (c5) to[out=180, in=180] (c8);
\draw[e] (c6) to[out=180, in=180] (c7);

\draw[e,gray] (c10) to [out=180,in=0] (1.75*\wid,9*\hei);

\draw[e] (c11) to[out=180, in=180] (c16);
\draw[e] (c12) to[out=180, in=180] (c13);
\draw[e] (c14) to[out=180, in=180] (c15);

\draw[e,gray] (c10) to[out=0, in=0] (c9);
\draw[e,gray] (c2) to[out=0, in=0] (c3);

\end{tikzpicture}
\quad
4:
\begin{tikzpicture}[x=1.5cm,y=-.5cm,baseline=-1.05cm]

\def\wid{2}
\def\hei{0.5}
\def\nodesize{3}
\def\ang{90}

\node[v, red, minimum size=\nodesize] (c1) at (2* \wid,0*\hei) {};
\node[v, gray, minimum size=\nodesize] (c2) at (2* \wid,1*\hei) {};
\node[v, gray, minimum size=\nodesize] (c3) at (2* \wid,2*\hei) {};
\node[v, gray, minimum size=\nodesize] (c4) at (2* \wid,3*\hei) {};
\node[v, gray, minimum size=\nodesize] (c5) at (2* \wid,4*\hei) {};
\node[v, minimum size=\nodesize] (c6) at (2* \wid,5*\hei) {};
\node[v, minimum size=\nodesize] (c7) at (2* \wid,6*\hei) {};
\node[v, red, minimum size=\nodesize] (c8) at (2* \wid,7*\hei) {};
\node[v, gray, minimum size=\nodesize] (c9) at (2* \wid,8*\hei) {};
\node[v, gray, minimum size=\nodesize] (c10) at (2* \wid,9*\hei) {};
\node[v, minimum size=\nodesize] (c11) at (2* \wid,10*\hei) {};
\node[v, minimum size=\nodesize] (c12) at (2* \wid,11*\hei) {};
\node[v, minimum size=\nodesize] (c13) at (2* \wid,12*\hei) {};
\node[v, minimum size=\nodesize] (c14) at (2* \wid,13*\hei) {};
\node[v, minimum size=\nodesize] (c15) at (2* \wid,14*\hei) {};
\node[v, minimum size=\nodesize] (c16) at (2* \wid,15*\hei) {};

\draw[e,gray] (c1) to [out=180,in=0] (1.75*\wid,0*\hei);

\draw[e,gray] (c2) to[out=180, in=180] (c9);
\draw[e,gray] (c3) to[out=180, in=180] (c4);
\draw[e,gray] (c5) to[out=180, in=180] (c8);
\draw[e] (c6) to[out=180, in=180] (c7);

\draw[e,gray] (c10) to [out=180,in=0] (1.75*\wid,9*\hei);

\draw[e] (c11) to[out=180, in=180] (c16);
\draw[e] (c12) to[out=180, in=180] (c13);
\draw[e] (c14) to[out=180, in=180] (c15);

\draw[e,gray] (c10) to[out=0, in=0] (c9);
\draw[e,gray] (c2) to[out=0, in=0] (c3);
\draw[e,gray] (c4) to[out=0, in=0] (c5);

\end{tikzpicture}
\quad
5:
\begin{tikzpicture}[x=1.5cm,y=-.5cm,baseline=-1.05cm]

\def\wid{2}
\def\hei{0.5}
\def\nodesize{3}
\def\ang{90}

\node[v, gray, minimum size=\nodesize] (c1) at (2* \wid,0*\hei) {};
\node[v, gray, minimum size=\nodesize] (c2) at (2* \wid,1*\hei) {};
\node[v, gray, minimum size=\nodesize] (c3) at (2* \wid,2*\hei) {};
\node[v, gray, minimum size=\nodesize] (c4) at (2* \wid,3*\hei) {};
\node[v, gray, minimum size=\nodesize] (c5) at (2* \wid,4*\hei) {};
\node[v, gray, minimum size=\nodesize] (c6) at (2* \wid,5*\hei) {};
\node[v, red, minimum size=\nodesize] (c7) at (2* \wid,6*\hei) {};
\node[v, red, minimum size=\nodesize] (c8) at (2* \wid,7*\hei) {};
\node[v, gray, minimum size=\nodesize] (c9) at (2* \wid,8*\hei) {};
\node[v, gray, minimum size=\nodesize] (c10) at (2* \wid,9*\hei) {};
\node[v, minimum size=\nodesize] (c11) at (2* \wid,10*\hei) {};
\node[v, minimum size=\nodesize] (c12) at (2* \wid,11*\hei) {};
\node[v, minimum size=\nodesize] (c13) at (2* \wid,12*\hei) {};
\node[v, minimum size=\nodesize] (c14) at (2* \wid,13*\hei) {};
\node[v, minimum size=\nodesize] (c15) at (2* \wid,14*\hei) {};
\node[v, minimum size=\nodesize] (c16) at (2* \wid,15*\hei) {};

\draw[e,gray] (c1) to [out=180,in=0] (1.75*\wid,0*\hei);

\draw[e,gray] (c2) to[out=180, in=180] (c9);
\draw[e,gray] (c3) to[out=180, in=180] (c4);
\draw[e,gray] (c5) to[out=180, in=180] (c8);
\draw[e,gray] (c6) to[out=180, in=180] (c7);

\draw[e,gray] (c10) to [out=180,in=0] (1.75*\wid,9*\hei);

\draw[e] (c11) to[out=180, in=180] (c16);
\draw[e] (c12) to[out=180, in=180] (c13);
\draw[e] (c14) to[out=180, in=180] (c15);

\draw[e,gray] (c10) to[out=0, in=0] (c9);
\draw[e,gray] (c2) to[out=0, in=0] (c3);
\draw[e,gray] (c4) to[out=0, in=0] (c5);
\draw[e,gray] (c1) to[out=0, in=0] (c6);

\end{tikzpicture}
\quad
6:
\begin{tikzpicture}[x=1.5cm,y=-.5cm,baseline=-1.05cm]

\def\wid{2}
\def\hei{0.5}
\def\nodesize{3}
\def\ang{90}

\node[v, gray, minimum size=\nodesize] (c1) at (2* \wid,0*\hei) {};
\node[v, gray, minimum size=\nodesize] (c2) at (2* \wid,1*\hei) {};
\node[v, gray, minimum size=\nodesize] (c3) at (2* \wid,2*\hei) {};
\node[v, gray, minimum size=\nodesize] (c4) at (2* \wid,3*\hei) {};
\node[v, gray, minimum size=\nodesize] (c5) at (2* \wid,4*\hei) {};
\node[v, gray, minimum size=\nodesize] (c6) at (2* \wid,5*\hei) {};
\node[v, red, minimum size=\nodesize] (c7) at (2* \wid,6*\hei) {};
\node[v, gray, minimum size=\nodesize] (c8) at (2* \wid,7*\hei) {};
\node[v, gray, minimum size=\nodesize] (c9) at (2* \wid,8*\hei) {};
\node[v, gray, minimum size=\nodesize] (c10) at (2* \wid,9*\hei) {};
\node[v, gray, minimum size=\nodesize] (c11) at (2* \wid,10*\hei) {};
\node[v, minimum size=\nodesize] (c12) at (2* \wid,11*\hei) {};
\node[v, minimum size=\nodesize] (c13) at (2* \wid,12*\hei) {};
\node[v, minimum size=\nodesize] (c14) at (2* \wid,13*\hei) {};
\node[v, minimum size=\nodesize] (c15) at (2* \wid,14*\hei) {};
\node[v, red, minimum size=\nodesize] (c16) at (2* \wid,15*\hei) {};

\draw[e,gray] (c1) to [out=180,in=0] (1.75*\wid,0*\hei);

\draw[e,gray] (c2) to[out=180, in=180] (c9);
\draw[e,gray] (c3) to[out=180, in=180] (c4);
\draw[e,gray] (c5) to[out=180, in=180] (c8);
\draw[e,gray] (c6) to[out=180, in=180] (c7);

\draw[e,gray] (c10) to [out=180,in=0] (1.75*\wid,9*\hei);

\draw[e,gray] (c11) to[out=180, in=180] (c16);
\draw[e] (c12) to[out=180, in=180] (c13);
\draw[e] (c14) to[out=180, in=180] (c15);

\draw[e,gray] (c10) to[out=0, in=0] (c9);
\draw[e,gray] (c2) to[out=0, in=0] (c3);
\draw[e,gray] (c4) to[out=0, in=0] (c5);
\draw[e,gray] (c1) to[out=0, in=0] (c6);
\draw[e,gray] (c8) to[out=0, in=0] (c11);

\end{tikzpicture}
\quad
\vspace{1cm}

7:
\begin{tikzpicture}[x=1.5cm,y=-.5cm,baseline=-1.05cm]

\def\wid{2}
\def\hei{0.5}
\def\nodesize{3}
\def\ang{90}

\node[v, gray, minimum size=\nodesize] (c1) at (2* \wid,0*\hei) {};
\node[v, gray, minimum size=\nodesize] (c2) at (2* \wid,1*\hei) {};
\node[v, gray, minimum size=\nodesize] (c3) at (2* \wid,2*\hei) {};
\node[v, gray, minimum size=\nodesize] (c4) at (2* \wid,3*\hei) {};
\node[v, gray, minimum size=\nodesize] (c5) at (2* \wid,4*\hei) {};
\node[v, gray, minimum size=\nodesize] (c6) at (2* \wid,5*\hei) {};
\node[v, gray, minimum size=\nodesize] (c7) at (2* \wid,6*\hei) {};
\node[v, gray, minimum size=\nodesize] (c8) at (2* \wid,7*\hei) {};
\node[v, gray, minimum size=\nodesize] (c9) at (2* \wid,8*\hei) {};
\node[v, gray, minimum size=\nodesize] (c10) at (2* \wid,9*\hei) {};
\node[v, gray, minimum size=\nodesize] (c11) at (2* \wid,10*\hei) {};
\node[v, minimum size=\nodesize] (c12) at (2* \wid,11*\hei) {};
\node[v, minimum size=\nodesize] (c13) at (2* \wid,12*\hei) {};
\node[v, minimum size=\nodesize] (c14) at (2* \wid,13*\hei) {};
\node[v, minimum size=\nodesize] (c15) at (2* \wid,14*\hei) {};
\node[v, red, minimum size=\nodesize] (c16) at (2* \wid,15*\hei) {};

\draw[e,gray] (c1) to [out=180,in=0] (1.75*\wid,0*\hei);

\draw[e,gray] (c2) to[out=180, in=180] (c9);
\draw[e,gray] (c3) to[out=180, in=180] (c4);
\draw[e,gray] (c5) to[out=180, in=180] (c8);
\draw[e,gray] (c6) to[out=180, in=180] (c7);

\draw[e,gray] (c10) to [out=180,in=0] (1.75*\wid,9*\hei);

\draw[e,gray] (c11) to[out=180, in=180] (c16);
\draw[e] (c12) to[out=180, in=180] (c13);
\draw[e] (c14) to[out=180, in=180] (c15);

\draw[e,gray] (c10) to[out=0, in=0] (c9);
\draw[e,gray] (c2) to[out=0, in=0] (c3);
\draw[e,gray] (c4) to[out=0, in=0] (c5);
\draw[e,gray] (c1) to[out=0, in=0] (c6);
\draw[e,gray] (c8) to[out=0, in=0] (c11);

\draw[e,gray] (c7) to [out=0,in=180] (2.25*\wid,6*\hei);

\end{tikzpicture}
\quad
8:
\begin{tikzpicture}[x=1.5cm,y=-.5cm,baseline=-1.05cm]

\def\wid{2}
\def\hei{0.5}
\def\nodesize{3}
\def\ang{90}

\node[v, gray, minimum size=\nodesize] (c1) at (2* \wid,0*\hei) {};
\node[v, gray, minimum size=\nodesize] (c2) at (2* \wid,1*\hei) {};
\node[v, gray, minimum size=\nodesize] (c3) at (2* \wid,2*\hei) {};
\node[v, gray, minimum size=\nodesize] (c4) at (2* \wid,3*\hei) {};
\node[v, gray, minimum size=\nodesize] (c5) at (2* \wid,4*\hei) {};
\node[v, gray, minimum size=\nodesize] (c6) at (2* \wid,5*\hei) {};
\node[v, gray, minimum size=\nodesize] (c7) at (2* \wid,6*\hei) {};
\node[v, gray, minimum size=\nodesize] (c8) at (2* \wid,7*\hei) {};
\node[v, gray, minimum size=\nodesize] (c9) at (2* \wid,8*\hei) {};
\node[v, gray, minimum size=\nodesize] (c10) at (2* \wid,9*\hei) {};
\node[v, gray, minimum size=\nodesize] (c11) at (2* \wid,10*\hei) {};
\node[v, minimum size=\nodesize] (c12) at (2* \wid,11*\hei) {};
\node[v, minimum size=\nodesize] (c13) at (2* \wid,12*\hei) {};
\node[v, red, minimum size=\nodesize] (c14) at (2* \wid,13*\hei) {};
\node[v, gray, minimum size=\nodesize] (c15) at (2* \wid,14*\hei) {};
\node[v, gray,  minimum size=\nodesize] (c16) at (2* \wid,15*\hei) {};

\draw[e,gray] (c1) to [out=180,in=0] (1.75*\wid,0*\hei);

\draw[e,gray] (c2) to[out=180, in=180] (c9);
\draw[e,gray] (c3) to[out=180, in=180] (c4);
\draw[e,gray] (c5) to[out=180, in=180] (c8);
\draw[e,gray] (c6) to[out=180, in=180] (c7);

\draw[e,gray] (c10) to [out=180,in=0] (1.75*\wid,9*\hei);

\draw[e,gray] (c11) to[out=180, in=180] (c16);
\draw[e] (c12) to[out=180, in=180] (c13);
\draw[e, gray] (c14) to[out=180, in=180] (c15);

\draw[e,gray] (c10) to[out=0, in=0] (c9);
\draw[e,gray] (c2) to[out=0, in=0] (c3);
\draw[e,gray] (c4) to[out=0, in=0] (c5);
\draw[e,gray] (c1) to[out=0, in=0] (c6);
\draw[e,gray] (c8) to[out=0, in=0] (c11);
\draw[e,gray] (c7) to [out=0,in=180] (2.25*\wid,6*\hei);
\draw[e,gray] (c16) to[out=0, in=0] (c15);

\end{tikzpicture}
\quad
9:
\begin{tikzpicture}[x=1.5cm,y=-.5cm,baseline=-1.05cm]

\def\wid{2}
\def\hei{0.5}
\def\nodesize{3}
\def\ang{90}

\node[v, gray, minimum size=\nodesize] (c1) at (2* \wid,0*\hei) {};
\node[v, gray, minimum size=\nodesize] (c2) at (2* \wid,1*\hei) {};
\node[v, gray, minimum size=\nodesize] (c3) at (2* \wid,2*\hei) {};
\node[v, gray, minimum size=\nodesize] (c4) at (2* \wid,3*\hei) {};
\node[v, gray, minimum size=\nodesize] (c5) at (2* \wid,4*\hei) {};
\node[v, gray, minimum size=\nodesize] (c6) at (2* \wid,5*\hei) {};
\node[v, gray, minimum size=\nodesize] (c7) at (2* \wid,6*\hei) {};
\node[v, gray, minimum size=\nodesize] (c8) at (2* \wid,7*\hei) {};
\node[v, gray, minimum size=\nodesize] (c9) at (2* \wid,8*\hei) {};
\node[v, gray, minimum size=\nodesize] (c10) at (2* \wid,9*\hei) {};
\node[v, gray, minimum size=\nodesize] (c11) at (2* \wid,10*\hei) {};
\node[v, red, minimum size=\nodesize] (c12) at (2* \wid,11*\hei) {};
\node[v, gray, minimum size=\nodesize] (c13) at (2* \wid,12*\hei) {};
\node[v, gray, minimum size=\nodesize] (c14) at (2* \wid,13*\hei) {};
\node[v, gray, minimum size=\nodesize] (c15) at (2* \wid,14*\hei) {};
\node[v, gray,  minimum size=\nodesize] (c16) at (2* \wid,15*\hei) {};

\draw[e,gray] (c1) to [out=180,in=0] (1.75*\wid,0*\hei);

\draw[e,gray] (c2) to[out=180, in=180] (c9);
\draw[e,gray] (c3) to[out=180, in=180] (c4);
\draw[e,gray] (c5) to[out=180, in=180] (c8);
\draw[e,gray] (c6) to[out=180, in=180] (c7);

\draw[e,gray] (c10) to [out=180,in=0] (1.75*\wid,9*\hei);

\draw[e,gray] (c11) to[out=180, in=180] (c16);
\draw[e, gray] (c12) to[out=180, in=180] (c13);
\draw[e, gray] (c14) to[out=180, in=180] (c15);

\draw[e,gray] (c10) to[out=0, in=0] (c9);
\draw[e,gray] (c2) to[out=0, in=0] (c3);
\draw[e,gray] (c4) to[out=0, in=0] (c5);
\draw[e,gray] (c1) to[out=0, in=0] (c6);
\draw[e,gray] (c8) to[out=0, in=0] (c11);
\draw[e,gray] (c7) to [out=0,in=180] (2.25*\wid,6*\hei);
\draw[e,gray] (c16) to[out=0, in=0] (c15);
\draw[e,gray] (c14) to[out=0, in=0] (c13);

\end{tikzpicture}
\quad
10:
\begin{tikzpicture}[x=1.5cm,y=-.5cm,baseline=-1.05cm]

\def\wid{2}
\def\hei{0.5}
\def\nodesize{3}
\def\ang{90}

\node[v, gray, minimum size=\nodesize] (c1) at (2* \wid,0*\hei) {};
\node[v, gray, minimum size=\nodesize] (c2) at (2* \wid,1*\hei) {};
\node[v, gray, minimum size=\nodesize] (c3) at (2* \wid,2*\hei) {};
\node[v, gray, minimum size=\nodesize] (c4) at (2* \wid,3*\hei) {};
\node[v, gray, minimum size=\nodesize] (c5) at (2* \wid,4*\hei) {};
\node[v, gray, minimum size=\nodesize] (c6) at (2* \wid,5*\hei) {};
\node[v, gray, minimum size=\nodesize] (c7) at (2* \wid,6*\hei) {};
\node[v, gray, minimum size=\nodesize] (c8) at (2* \wid,7*\hei) {};
\node[v, gray, minimum size=\nodesize] (c9) at (2* \wid,8*\hei) {};
\node[v, gray, minimum size=\nodesize] (c10) at (2* \wid,9*\hei) {};
\node[v, gray, minimum size=\nodesize] (c11) at (2* \wid,10*\hei) {};
\node[v, gray, minimum size=\nodesize] (c12) at (2* \wid,11*\hei) {};
\node[v, gray, minimum size=\nodesize] (c13) at (2* \wid,12*\hei) {};
\node[v, gray, minimum size=\nodesize] (c14) at (2* \wid,13*\hei) {};
\node[v, gray, minimum size=\nodesize] (c15) at (2* \wid,14*\hei) {};
\node[v, gray,  minimum size=\nodesize] (c16) at (2* \wid,15*\hei) {};

\draw[e,gray] (c1) to [out=180,in=0] (1.75*\wid,0*\hei);

\draw[e,gray] (c2) to[out=180, in=180] (c9);
\draw[e,gray] (c3) to[out=180, in=180] (c4);
\draw[e,gray] (c5) to[out=180, in=180] (c8);
\draw[e,gray] (c6) to[out=180, in=180] (c7);

\draw[e,gray] (c10) to [out=180,in=0] (1.75*\wid,9*\hei);

\draw[e,gray] (c11) to[out=180, in=180] (c16);
\draw[e, gray] (c12) to[out=180, in=180] (c13);
\draw[e, gray] (c14) to[out=180, in=180] (c15);

\draw[e,gray] (c10) to[out=0, in=0] (c9);
\draw[e,gray] (c2) to[out=0, in=0] (c3);
\draw[e,gray] (c4) to[out=0, in=0] (c5);
\draw[e,gray] (c1) to[out=0, in=0] (c6);
\draw[e,gray] (c8) to[out=0, in=0] (c11);
\draw[e,gray] (c7) to [out=0,in=180] (2.25*\wid,6*\hei);
\draw[e,gray] (c16) to[out=0, in=0] (c15);
\draw[e,gray] (c14) to[out=0, in=0] (c13);
\draw[e,gray] (c12) to [out=0,in=180] (2.25*\wid,11*\hei);

\end{tikzpicture}
\quad
\end{center}
so
\begin{center}
$p=$\quad
\begin{tikzpicture}[x=1.5cm,y=-.5cm,baseline=-1.05cm]

\def\wid{2}
\def\hei{0.5}
\def\nodesize{3}
\def\ang{90}

\node[v, minimum size=\nodesize] (c1) at (2* \wid,0*\hei) {};
\node[v, minimum size=\nodesize] (c2) at (2* \wid,1*\hei) {};
\node[v, minimum size=\nodesize] (c3) at (2* \wid,2*\hei) {};
\node[v, minimum size=\nodesize] (c4) at (2* \wid,3*\hei) {};
\node[v, minimum size=\nodesize] (c5) at (2* \wid,4*\hei) {};
\node[v, minimum size=\nodesize] (c6) at (2* \wid,5*\hei) {};
\node[v, minimum size=\nodesize] (c7) at (2* \wid,6*\hei) {};
\node[v, minimum size=\nodesize] (c8) at (2* \wid,7*\hei) {};
\node[v, minimum size=\nodesize] (c9) at (2* \wid,8*\hei) {};
\node[v, minimum size=\nodesize] (c10) at (2* \wid,9*\hei) {};
\node[v, minimum size=\nodesize] (c11) at (2* \wid,10*\hei) {};
\node[v, minimum size=\nodesize] (c12) at (2* \wid,11*\hei) {};
\node[v, minimum size=\nodesize] (c13) at (2* \wid,12*\hei) {};
\node[v, minimum size=\nodesize] (c14) at (2* \wid,13*\hei) {};
\node[v, minimum size=\nodesize] (c15) at (2* \wid,14*\hei) {};
\node[v, minimum size=\nodesize] (c16) at (2* \wid,15*\hei) {};

\draw[e] (c10) to[out=180, in=180] (c9);
\draw[e] (c2) to[out=180, in=180] (c3);
\draw[e] (c4) to[out=180, in=180] (c5);
\draw[e] (c1) to[out=180, in=180] (c6);
\draw[e] (c8) to[out=180, in=180] (c11);
\draw[e] (c7) to [out=180,in=0] (1.75*\wid,6*\hei);
\draw[e] (c16) to[out=180, in=180] (c15);
\draw[e] (c14) to[out=180, in=180] (c13);
\draw[e] (c12) to [out=180,in=0] (1.75*\wid,11*\hei);

\end{tikzpicture}
\quad
\end{center}
is as required.
\end{example}

\begin{proposition} \label{prop:TLIdempotents}
Let $t \geq 1$, $q \in M(t)$. The left ideal $J_{\leq q}$ is principal and generated by an idempotent.
\end{proposition}

\begin{proof} By Lemma \ref{lem:TLLSOGenerating}, the link state ordering on the Temperley-Lieb algebras is left generating. Since the Temperley-Lieb algebras are diagram-like (Example \ref{ex:TL}), we may attempt to use Proposition \ref{prop:Idempotents}.

To do so, we must construct $v \in W(t)$ such that $$\langle C_{q}, v \rangle_{\tau} = \begin{cases} 1 & \textrm{ if } \tau=\sigma, \textrm{ and} \\ 0 & \textrm{ otherwise.} \end{cases}$$

Since for all $t$, the group $G(t)$ in the naive-cellular datum is the trivial group (i.e. the naive-cellular structure on the Temperley-Lieb algebras is just Graham and Lehrer's cellular structure), this reduces to constructing $v$ such that $\langle C_{q}, v \rangle = 1$, where the bilinear form is Graham and Lehrer's (c.f Remark \ref{rmk:cellForm}).

By Lemma \ref{lem:formInterpretation}, it suffices to take $v$ to be the element $C_p$ provided by Proposition \ref{prop:wiggleWiggle}. The result follows. \end{proof}

We are now ready to give a variant proof of Sroka's theorem \cite{Sroka}. It is not a completely new proof because the topological `back-end' (in the guise of Theorem 1.7 of \cite{Boyde2}) is essentially unchanged, and the `front-end' has just been dressed up in the language of cellular algebras, to permit a `bilinear form' verification of the fact that the $\bigcap_{i \in S} K_i$ are principal ideals generated by idempotents.

Sroka's result is as follows:

\begin{theorem} For all commutative rings $R$, all $\delta \in R$, and $n \geq 0$, we have $$\Tor_q^{\TL_n(\delta)}(\t,\t) \cong \begin{cases}
    R & q = 0 \textrm{ and} \\
    0 & 0 < q \leq \frac{n}{2}-2.
\end{cases}$$
Furthermore, if $n$ is odd, then actually $\Tor_q^{\TL_n(\delta)}(\t,\t) \cong 0$ for all $q > 0$.
\end{theorem}

\begin{proof} By Lemma \ref{lem:intersectionDescription}, Lemma \ref{lem:intersectionsToJ}, and Proposition \ref{prop:TLIdempotents}, the ideals $K_i$ form a principal idempotent cover \cite[Definition 1.6]{Boyde2} of $I_{\leq n-1}$, and this cover has height $\frac{n}{2}-1$ if $n$ is even, and height $n-1$ if $n$ is odd.

Then apply Theorem \ref{oldThm}, together with the observations that $\faktor{\TL_n(\delta)}{I_{\leq n-1}} \cong R$, and $$\Tor^R_q(\t,\t) \cong \begin{cases}
    0 & q > 0 \textrm{ and } \\
    R & q = 0
\end{cases}$$ to complete the proof. \end{proof}

\printbibliography

\end{document}